\documentclass[11pt, reqno]{amsart}

\usepackage{amsmath, amsthm, amssymb}
\usepackage{enumerate}
\usepackage[margin=1.5in]{geometry}
\usepackage{tikz}
\usepackage{tikz-cd}
\usepackage{ytableau}
\usepackage{arydshln}
\usepackage{graphicx}
\usepackage{mathdots}
\usepackage[colorlinks=true,urlcolor=blue,linkcolor=blue,citecolor=blue]{hyperref}
\usepackage{cleveref}
\newtheorem{theorem}{Theorem}[section]
\newtheorem{proposition}[theorem]{Proposition}
\newtheorem{lemma}[theorem]{Lemma}

\newtheorem{corollary}[theorem]{Corollary}

\theoremstyle{remark}
\newtheorem{remark}[theorem]{Remark}
\newtheorem*{claim}{Claim}
\newtheorem{example}[theorem]{Example}

\theoremstyle{definition}
\newtheorem{definition}[theorem]{Definition}
\newtheorem*{question}{Basic Question}

\newcommand{\N}{\mathbb{N}}
\newcommand{\Z}{\mathbb{Z}}
\newcommand{\Q}{\mathbb{Q}}
\newcommand{\R}{\mathbb{R}}
\newcommand{\C}{\mathbb{C}}

\newcommand{\SL}{\mathrm{SL}}
\newcommand{\GL}{\mathrm{GL}}
\newcommand{\SO}{\mathrm{SO}}

\numberwithin{equation}{section}

\title[Expanding Translates of Curves]{Equidistribution of Expanding Translates of Curves and Diophantine Approximation on Matrices}
\author{Pengyu Yang}
\address{Department of Mathematics\\
	The Ohio State University\\
	Columbus, OH}
\email{yang.2214@osu.edu}
\date{September 21, 2018}
\keywords{Geometric invariant theory, homogeneous spaces, equidistribution, Ratner's theorem,  Dirichlet's theorem, Diophantine approximation}
\subjclass[2010]{22E40; 14L24; 11J83}

\begin{document}
	
	\begin{abstract}
		We study the general problem of equidistribution of expanding translates of an analytic curve by an algebraic diagonal flow on the homogeneous space $G/\Gamma$ of a semisimple algebraic group $G$. We define two families of algebraic subvarieties of the associated partial flag variety $G/P$, which give the obstructions to non-divergence and equidistribution. We apply this to prove that for Lebesgue almost every point on an analytic curve in the space of $m\times n$ real matrices whose image is not contained in any subvariety coming from these two families, the Dirichlet's theorem on simultaneous Diophantine approximation cannot be improved.\par 
		The proof combines geometric invariant theory, Ratner's theorem on measure rigidity for unipotent flows, and linearization technique.
	\end{abstract}\maketitle
	\setcounter{tocdepth}{1}
	\tableofcontents
	
	\section{Introduction}
	\subsection{Background}
	Many problems in number theory can be recast in the language of homogeneous dynamics. Let $G$ be a Lie group and $\Gamma$ be a lattice in $G$, i.e. a discrete subgroup of finite covolume. Take a sequence $\{g_i\}$ in $G$ and a probability measure $\mu$ on $G/\Gamma$ which is supported on a smooth submanifold of $G/\Gamma$. The following question was raised by Margulis in \cite{Mar02}:
	
	\begin{question}[Margulis]
		What is the distribution of $g_i\mu$ in $G/\Gamma$ when $g_i$ tends to infinity in $G$?
	\end{question}
	
	In 1993, Duke, Rudnick and Sarnak \cite{DRS93} studied the case where $\mu$ is a finite invariant measure supported on a symmetric subgroup orbit, and applied it to obtain asymptotic estimates for the number of integral points of bounded norm on affine symmetric varieties. At the same time, Eskin and McMullen \cite{EM93} gave a simpler proof using the mixing property of geodesic flows. It was later generalized by Eskin, Mozes and Shah \cite{EMS96} to the case where $\mu$ is a finite invariant measure supported on a reductive group orbit, and applied it to count integral matrices of bounded norm with a given characteristic polynomial. Later Gorodnik and Oh \cite{GO11} worked in the Adelic setting, and gave an asymptotic formula for the number rational points of bounded height on homogeneous varieties. \par 
	
	In another direction, the dynamical behavior of translates of a submanifold of expanding horospherical subgroups in $\SL_n(\R)/\SL_n(\Z)$ is closely related to metric Diophantine approximation. In 1998, Kleinbock and Margulis \cite{KM98} proved extremality of a non-degenerate submanifold in $\R^n$, and their proof was based on quantitative non-divergence of translates of the submanifold by semisimple elements. Their work was later extended from $\R^n$ to the space $M_{m\times n}(\R)$ of $m\times n$ real matrices (see e.g. \cite{KMW10}\cite{BKM15}\cite{ABRS18}).\par 
	
	While quantitative non-divergence results are useful in the study of extremality, equidistribution results can be applied to study the improvability of Dirichlet's theorem. In 2008, Kleinbock and Weiss \cite{KW08} first explored improvability in the language of homogeneous dynamics, based on earlier observations by Dani \cite{Dan84} as well as Kleinbock and Margulis \cite{KM98}. Later Shah \cite{Sha09Invention} obtained a strengthened result for analytic curves in $\R^n$ by showing equidistribution of expanding translates of curves in $\SL_{n+1}(\R)/\SL_{n+1}(\Z)$ by singular diagonal elements $a(t)=\mathrm{diag}(t^{n},t^{-1},\cdots, t^{-1})$. This work has also been generalized to $m\times n$ matrices in a recent preprint \cite{SY16} by Shah and Lei Yang, where they considered the case $G=\SL_{m+n}(\R)$ and $a(t)=\mathrm{diag}(t^n,\cdots,t^n,t^{-m},\cdots,t^{-m})$. We shall discuss this subject in more details in \Cref{subsect:Grassmannian_Dirichlet}. \par 
	
	It is also worth considering the case $G=\SO(n,1)$, as there are interesting applications to hyperbolic geometry. See Shah's works \cite{Sha09Duke1}\cite{Sha09Duke2} and later generalizations by Lei Yang \cite{Yan16Israel}\cite{Yan17}. We shall provide more details in \Cref{subsect:equidistribution}. \\
	
	Motivated by the previous works, we are interested in the following equidistribution problem, which was proposed by Shah in ICM 2010 \cite{Sha10PICM}. Let $G=\mathbf{G}(\R)$ be a semisimple connected real algebraic group of non-compact type, and let $L$ be a Lie group containing $G$. Let $\Lambda$ be a lattice in $L$. Let $\{a(t)\}_{t\in\R^{\times}}$ be a multiplicative one-parameter subgroup of $G$, i.e. we have a homomorphism of real algebraic group $a \colon \mathbb{G}_m \rightarrow \mathbf{G}$. Suppose we have a bounded piece of an analytic curve on $G$ given by $\phi\colon I=[a,b]\rightarrow G$, and we fix a point $x_0$ on $L/\Lambda$ such that $Gx_0$ is dense in $L/\Lambda$. Let $\lambda_{\phi}$ denote the measure on $L/\Lambda$ which is the parametric measure supported on the orbit $\phi(I)x_0$, that is, $\lambda_{\phi}$ is the pushforward of the Lebesgue measure. When does $a(t)\lambda_{\phi}$ converge to the Haar measure on $L/\Lambda$ with respect to the weak-* topology, as $t$ tends to infinity? \par
	
	In \cite{Sha10PICM}, Shah found natural algebraic obstructions to equidistribution, and asked if those are the only obstructions. In this article, we give an affirmative answer to Shah's question. This generalizes previous results on $G=\SO(n,1)$ \cite{Sha09Duke1}\cite{Yan16Israel}, $G=\SO(n,1)^k$ \cite{Yan17}, as well as $G=\SL_n(\R)$ and $a(t)$ being singular \cite{Sha09Invention}\cite{SY16}. We also apply the equidistribution result to show that for almost every point on a ``non-degenerate'' analytic curve in the space of $m\times n$ real matrices, the Dirichlet's theorem cannot be improved. This sharpens a result of Shah and Yang \cite{SY16}. \par 
	
	We remark that our method also applies to analytic submanifolds. For convenience, we restrict our discussions to curves.
	
	\subsection{Non-escape of mass to infinity}\label{subsect:non_escape_of_mass}
	Let $G=\mathbf{G}(\R)$ be a semisimple connected real algebraic group of non-compact type, and let $L$ be a Lie group containing $G$. Let $\Lambda$ be a lattice in $L$. Let $\{a(t)\}_{t\in\R^{\times}}$ be a multiplicative one-parameter subgroup of $G$ with non-trivial projection on each simple factor of $G$. There is a parabolic subgroup $P=P(a)$ of $G$ associated with $a(t)$: 
	\begin{equation}
	P := \{ g \in G \colon \lim_{t\to\infty} a(t)ga(t)^{-1} \text{ exists in $G$}\}.
	\end{equation}
	Suppose we have a bounded piece of an analytic curve on $G$ given by $\phi:I=[a,b]\rightarrow G$, and we fix a point $x_0$ on $L/\Lambda$ such that the orbit $Gx_0$ is dense in $L/\Lambda$. Let $\lambda_{\phi}$ denote the parametric measure on $L/\Lambda$. If we expect the translated measures to get equidistributed, it is necessary that there is no escape of mass to infinity. \par 
	Let us first consider the special case $G=L=\SL_{m+n}(\R)$, $\Lambda=\SL_{m+n}(\Z)$ and $a(t)=\mathrm{diag}(t^n,\cdots,t^n,t^{-m},\cdots,t^{-m})$. In \cite{ABRS18}, Aka, Breuillard, Rosenzweig and de Saxc\'e defined a family of algebraic sets called \emph{constraining pencils} (see \cite[Definition 1.1]{ABRS18}), and used it to describe the obstruction to quantitative non-divergence. They remarked that constraining pencils give rise to certain Schubert varieties in Grassmannians.\par 
	Inspired by their work, we define the notion of \emph{unstable Schubert varieties}\footnote{The name comes from the notion of stability in geometric invariant theory, and should not be confused with unstable manifolds for a diffeomorphism.} (see \Cref{def:Schubert_variety}) with respect to $a(t)$ for general partial flag variety $G/P$, which naturally generalizes the notion of constraining pencils. This enables us to describe obstructions to non-divergence in general case. \par 
	Now we project our curve $\phi$ onto $G/P$. Consider 
	\begin{align}\label{eq:phi_tilde}
	\widetilde{\phi}\colon&[a,b]\longrightarrow G/P \nonumber \\
	&s \longmapsto \phi(s)^{-1}P. 
	\end{align}
	We are taking inverse here simply because we would like to quotient $P$ on the right, which is the case in most literatures.\par 
	We are ready to state our first main theorem on non-escape of mass.
	
	\begin{theorem}[Non-escape of mass]\label{thm:non_divergence_introduction}	
		Let $\phi:I=[a,b]\rightarrow G$ be an analytic curve such that the image of $\widetilde{\phi}$ is not contained in any unstable Schubert variety of $G/P$ with respect to $a(t)$. Then for any $\epsilon > 0$, there exists a compact subset $K$ of $L/\Lambda$ such that for any $t > 1$, we have
		\begin{equation}
		\frac{1}{b-a}\left\vert\left\{ s\in[a,b]\colon a(t)\phi(s)x_0\in K \right\} \right\vert > 1 - \epsilon.
		\end{equation}
	\end{theorem}
	To prove \Cref{thm:non_divergence_introduction}, we consider a certain finite dimensional representation $V$ of $G$ (see \Cref{def:V}), and show that the corresponding curve in $V$ cannot be uniformly contracted to the origin. The key ingredient is the following theorem, which is the main technical contribution of this article.
	
	\begin{theorem}[Linear stability]\label{thm:linear_stability}
		Let $\rho\colon G\rightarrow \GL(V)$ be any finite-dimensional linear representation of $G$, with a norm $\lVert\cdot\rVert$ on $V$. Suppose that the image of $\widetilde{\phi}$ is not contained in any unstable Schubert variety of $G/P$ with respect to $a(t)$. Then there exists a constant $C>0$ such that for any $t>1$ and any $v\in V$, one has
		\begin{equation}\label{eq:linear_stablility}
		\sup_{s\in [a,b]}\lVert a(t)\phi(s)v \rVert \geq C\lVert v \rVert.
		\end{equation}
	\end{theorem}
	
	\Cref{thm:linear_stability} is of independent interest, as it is also applicable to obtain quantitative non-divergence results (see e.g. \cite{Shi15}). Compared to the previous works on special cases of the theorem, the novel part of our proof is that we use a result in geometric invariant theory, which is Kempf's numerical criterion \cite[Theorem 4.2]{Kem78}. \par 
	Geometric invariant theory was first developed by Mumford to construct quotient varieties in algebraic geometry; its connections to dynamics have been found in recent years. Kapovich, Leeb and Porti \cite[Section 7.4]{KLP18} explored the relation with geometric invariant theory for groups of type $A_1^n$. In a recent preprint \cite{Kha15}, Khayutin utilized geometric invariant theory to study the double quotient of a reductive group by a torus. In \cite[Section 6]{RS16}, Richard and Shah applied \cite[Lemma 1.1(b)]{Kem78} to deal with focusing, which also came from the study of geometric invariant theory. \par 
	\Cref{thm:linear_stability} is proved in \Cref{sect:linear_stability}, and \Cref{thm:non_divergence_introduction} is proved in \Cref{sect:nondivergence}.

	\subsection{Equidistribution of translated measures}\label{subsect:equidistribution}
	Let the notations be as in \Cref{subsect:non_escape_of_mass}, and suppose that the image of $\widetilde{\phi}\colon s\mapsto\phi(s)^{-1}P$ is not contained in any unstable Schubert variety of $G/P$ with respect to $a(t)$ (see \Cref{def:Schubert_variety}). Due to \Cref{thm:non_divergence_introduction}, for any sequence $t_i \to\infty$, the sequence of translated measures $a(t)\lambda_{\phi}$ is tight, i.e. any weak-* limit is a probability measure on $L/\Lambda$. If one can further show that any limit measure is the Haar measure on $L/\Lambda$, then the translated measure $a(t)\lambda_{\phi}$ gets equidistributed as $t\to\infty$. In order to achieve this, one needs to exclude a larger family of obstructions. \par 
	In a sequence of papers \cite{Sha09Duke1}\cite{Sha09Duke2}\cite{Sha09Invention}, Shah initiated the study of the curve equidistribution problem with several important special cases. For example, when $G=\SL_{n+1}(\R)$ and $a(t)=\mathrm{diag}(t^{n},t^{-1},\cdots,t^{-1})$, the obstructions to equidistribution come from linear subspaces of $\R\mathbb{P}^n$, which are exactly the unstable Schubert varieties with respect to $a(t)$.\par 
	Another interesting case is when $G=\SO(n,1)$ and $\{a(t)\}$ being the geodesic flow on the unit tangent bundle $T^1(\mathbb{H}^n)$ of the hyperbolic space $\mathbb{H}^n\cong \SO(n,1)/\SO(n)$. The visual boundary of $\mathbb{H}^n$ has the identification
	\begin{equation}
	\partial\mathbb{H}^n\cong\mathbb{S}^{n-1}\cong G/P.
	\end{equation}
	Shah found that the obstructions to equidistribution comes from proper subspheres $\mathbb{S}^{m-1}$ of $\mathbb{S}^{n-1}$ ($m<n$). However, since the real rank of $G$ is one, the proper Schubert varieties of $G/P$ are just single points. Therefore, these obstructions are not given by Schubert varieties. Nonetheless, the subspheres are still natural geometric objects, as they are closed orbits of the subgroups $\SO(m,1)\subset\SO(n,1)$, which correspond to totally geodesic submanifolds $\mathbb{H}^m\subset\mathbb{H}^n$. \par 
	Motivated by these results, Shah \cite{Sha10PICM} found the following algebraic obstruction to equidistribution in the general setting. Suppose that $F$ is a proper subgroup of $L$ containing $\{a(t)\}$, and $g\in G$ is an element such that the orbit $Fgx_0$ is closed and carries a finite $F$-invariant measure. Suppose that $\phi(I)\subset P(F\cap G)g$. Then for any sequence $t_i\to\infty$, it follows that any weak-* limit of probability measures $a(t_i)\lambda_{\phi}$ is a direct integral of measures which are supported on closed sets of the form $bFgx_0$, where $b\in P$. Such limiting measures are concentrated on strictly lower dimensional submanifolds of $L/\Lambda$. Shah also asked if these are the only obstructions. \par 
	We now state our main theorem on equidistribution, which answers Shah's question affirmatively. Recall that $x_0$ is an element in $L/\Lambda$ such that $Gx_0$ is dense in $L/\Lambda$. Let $\widetilde{\phi}$ be as in \eqref{eq:phi_tilde}. For the definition of \emph{unstable Schubert variety}, see \Cref{def:Schubert_variety}.
	
	\begin{theorem}\label{thm:equidistribution0}
		Let $\phi\colon I=[a,b]\rightarrow G$ be an analytic curve such that the following two conditions hold:
		\begin{enumerate}[(a)]
			\item The image of $\widetilde{\phi}$ is not contained in any unstable Schubert variety of $G/P$ with respect to $a(t)$;
			\item For any $g\in G$ and any proper algebraic subgroup $F$ of $L$ containing $\{a(t)\}$ such that $Fgx_0$ is closed and admits a finite $F$-invariant measure, the image of $\phi$ is not contained in $P(F\cap G)g$.
		\end{enumerate}
		Then for any $f\in C_c(L/\Lambda)$, we have
		\begin{equation}\label{eq:equidistribution0}
		\lim_{t\to\infty}\frac{1}{b-a}\int_{a}^{b}f(a(t)\phi(s)x_0)\,\mathrm{d}s = \int_{L/\Lambda}f\,\mathrm{d}\mu_{L/\Lambda},
		\end{equation}
		where $\mu_{L/\Lambda}$ is the $L$-invariant probability measure on $L/\Lambda$.
	\end{theorem}
	
	\begin{remark}
		In \Cref{thm:equidistribution0}, if we assume $(a)$ holds, then by the above discussion we know that \eqref{eq:equidistribution0} holds if and only if $(b)$ holds. In this sense, our result is sharp.
	\end{remark}
	
	One can even require $F\cap G$ to be reductive if we replace the family of unstable Schubert varieties with the slightly larger family of \emph{weakly unstable Schubert varieties} (see \Cref{def:Schubert_variety}).
	
	\begin{theorem}\label{thm:equidistribution1}
		Let $\phi\colon I=[a,b]\rightarrow G$ be an analytic curve such that the following two conditions hold:
		\begin{enumerate}[(A)]
			\item The image of $\widetilde{\phi}$ is not contained in any weakly unstable Schubert variety of $G/P$ with respect to $a(t)$;
			\item For any $g\in G$ and any proper algebraic subgroup $F$ of $L$ containing $\{a(t)\}$ such that $Fgx_0$ is closed and admits a finite $F$-invariant measure and that $F\cap G$ is reductive, the image of $\phi$ is not contained in $P(F\cap G)g$.
		\end{enumerate}
		Then for any $f\in C_c(L/\Lambda)$, we have
		\begin{equation}\label{eq:equidistribution1}
		\lim_{t\to\infty}\frac{1}{b-a}\int_{a}^{b}f(a(t)\phi(s)x_0)\,\mathrm{d}s = \int_{L/\Lambda}f\,\mathrm{d}\mu_{L/\Lambda},
		\end{equation}
		where $\mu_{L/\Lambda}$ is the $L$-invariant probability measure on $L/\Lambda$.
	\end{theorem}
	
	If a reductive subgroup $H$ contains $\{a(t)\}$, then $P_H=P\cap H$ is a parabolic subgroup of $H$ associated with $a(t)$, and $HP/P$ is homeomorphic to $H/P_H$. Hence we give the following definition.
	
	\begin{definition}[Partial flag subvariety]\label{def:partial_flag_subvariety}
		A \emph{partial flag subvariety of $G/P$ with respect to $a(t)$} is a subvariety of the form $gHP/P$, where $g$ is an element in $G$, and $H$ is a reductive subgroup of $G$ containing $\{a(t)\}$.
	\end{definition}
	
	In view of \Cref{def:partial_flag_subvariety}, \Cref{thm:equidistribution1} shows that the obstructions consist of two families of geometric objects: weakly unstable Schubert varieties and partial flag subvarieties.\par 

	\Cref{thm:equidistribution0} and \Cref{thm:equidistribution1} are proved in \Cref{sect:focusing}.
	
	\subsection{Grassmannians and Dirichlet's approximation theorem on matrices}\label{subsect:Grassmannian_Dirichlet}
	In this section, we give an application of our equidistribution result to simultaneous Diophantine approximation.\par 
	In 1842, Dirichlet proved a theorem on simultaneous approximation of a matrix of real numbers (DT): Given any two positive integers $m$ and $n$, a matrix $\Psi\in M_{m\times n}(\R)$, and $N>0$, there exist integral vectors $\mathbf{p}\in\Z^n\backslash\{\mathbf{0}\}$ and $\mathbf{q}\in\Z^m$ such that
	\begin{equation}
	\lVert\mathbf{p}\rVert\leq N^m \quad \text{and} \quad \lVert\Psi\mathbf{p}-\mathbf{q}\rVert \leq N^{-n},
	\end{equation}
	where $\lVert\cdot\rVert$ denotes the supremum norm, that is, $\lVert x\rVert=\max_{1\leq i\leq k}\lvert x_i\rvert$ for any $\mathbf{x}=(x_1,x_2,\cdots,x_k)\in\R^k.$
	
	Given $0<\mu<1$. After Davenport and Schmidt \cite{DS6970}, we say that $\Psi\in\mathrm{M}_{m\times n}(\R)$ is $\mathrm{DT}_\mu$-improvable if for all sufficiently large $N>0$, there exists nonzero integer vectors $\mathbf{p}\in\Z^n$ and $\mathbf{q}\in\Z^m$ such that
	\begin{equation}
	\lVert\mathbf{p}\rVert\leq \mu N^m \quad \text{and} \quad \lVert\Psi\mathbf{p}-\mathbf{q}\rVert \leq \mu N^{-n}.
	\end{equation}
	We say that $\Psi$ is not DT-improvable, if for any $0<\mu<1$, $\Psi$ is not $\mathrm{DT}_\mu$-improvable.\par 
	In \cite{DS6970}, it was proved that Dirichlet's theorem cannot be improved for Lebesgue almost every $m \times n$ real matrix. In \cite{DS70}, they also proved that Dirichlet's theorem cannot be $(1/4)$-improved for almost every point on the curve $\phi(s)=(s,s^2)$ in $\mathbb{R}^2$. This result was generalized by Baker \cite{Bak78} for almost all points on smooth curves in $\mathbb{R}^2$, and by Bugeaud \cite{Bug02} for almost every point on the curve $\phi(s)=(s, s^2, \cdots, s^k)$ in $\mathbb{R}^k$; in each case the result holds for some small value $0<\mu \leq \epsilon$, where $\epsilon$ depends on the curve.\par
	Kleinbock and Weiss \cite{KW08} recast the problem in the language of homogeneous dynamics, and obtained $\epsilon$-improvable results for general measures. Later Shah \cite{Sha09Invention} studied the case $m=1$, and showed that if an analytic curve in $\R^n$ is not contained in any proper affine subspace, then almost every point on the curve is not DT-improvable. Lei Yang \cite{Yan16Proc} studied the case $m=n$, and proved an analogous result for square matrices. These results have been generalized to \emph{supergeneric} curves in $M_{m\times n}(\R)$ in the recent preprint \cite{SY16}, where an inductive algorithm was introduced to define \emph{generic} and \emph{supergeneric} curves. \par 
	In the meantime, Aka, Breuillard, Rosenzweig and de Saxc\'e \cite{ABRS18} worked on extremality of an analytic submanifold of $M_{m\times n}(\R)$, and found a sharp condition for extremality in terms of a certain family of algebraic sets called \emph{constraining pencils} (see \cite[Definition 1.1]{ABRS18}).\par 
	Based on \cite{SY16}, and combined with ideas from \cite{Sha09Duke1}\cite{ABRS18}, we replace supergeneric condition by a natural geometric condition, and obtain a sharper result. \par 
	We first make some preparations. Let $\mathrm{Gr}(m,m+n)$ denote the real Grassmannian variety of $m$-dimensional linear subspaces of $\R^{m+n}$.
	
	\begin{definition}[pencil; c.f. \cite{ABRS18} Definition 1.1]\label{def:pencil}
		Given a real vector space $W \subsetneq \R^{m+n}$, and an integer $r\leq m$, we define the pencil $\mathfrak{P}_{W,r}$ to be the set 
		\begin{equation}
		\{V\in\mathrm{Gr}(m,m+n)\colon \mathrm{dim}(V\cap W)\geq r\}.
		\end{equation}
	\end{definition}
	We call $\mathfrak{P}_{W,r}$ a \emph{constraining pencil} if
	\begin{equation}\label{eq:constraining_pencil}
	\frac{\dim W}{r} < \frac{m+n}{m};
	\end{equation}
	we call $\mathfrak{P}_{W,r}$ a \emph{weakly constraining pencil} if
	\begin{equation}\label{eq:weakly_constraining_pencil}
	\frac{\dim W}{r} \leq \frac{m+n}{m}.
	\end{equation}
	We say that the pencil $\mathfrak{P}_{W,r}$ is \emph{rational} if $W$ is rational, i.e. $W$ admits a basis in $\Q^{m+n}$.
	
	\begin{remark}\label{rem:pencils}
		\begin{enumerate}[(1)]
			\item If $m$ and $n$ are coprime, then $\frac{m+n}{m}$ is an irreducible fraction, and it follows that \eqref{eq:weakly_constraining_pencil} and \eqref{eq:constraining_pencil} are equivalent. Therefore weakly constraining pencils coincide with constraining pencils in this case.
			\item If $m=1$, then (weakly) constraining pencils are proper linear subspaces of $\R\mathbb{P}^n$.
		\end{enumerate}
	\end{remark}
	
	To avoid confusions, we explain the relationship between our pencils and the pencils in \cite{ABRS18}. Given $W \subsetneq \R^{m+n}$ and $0<r<m$, in \cite{ABRS18} a pencil $\mathcal{P}_{W,r}$ is defined to be an algebraic subset of $M_{m\times(m+n)}(\R)$. More precisely,
	\begin{equation}
	\mathcal{P}_{W,r} = \left\{ x\in M_{m\times(m+n)}(\R)\colon \dim(xW)\leq r \right\}.
	\end{equation}
	And a pencil $\mathcal{P}_{W,r}$ is called constraining if
	\begin{equation}
	\frac{\dim W}{r} > \frac{m+n}{m}.
	\end{equation}
	Let $x$ be a full rank $m\times(m+n)$ real matrix. For any subspace $E\subset\R^{m+n}$, let $E^{\vee}\subset(\R^{m+n})^{\ast}$ denote the set of linear functionals on $\R^{m+n}$ which vanish on $E$. Then $\dim(xW)\leq r$ if and only if $\dim\left((\ker x)^{\vee}\cap W^{\vee} \right) \geq m-r$. Hence
	\begin{equation}
	x\in\mathcal{P}_{W,r}\quad\iff\quad (\ker x)^{\vee}\in\mathfrak{P}_{W^{\vee},m-r}.
	\end{equation}
	Moreover, since $\dim W^{\vee} = m+n-\dim W$, we have
	\begin{equation}
	\frac{\dim W}{r} > \frac{m+n}{m} \quad \iff \quad \frac{\dim W^{\vee}}{m-r} < \frac{m+n}{m}.
	\end{equation}
	As explained in \cite[Section 4]{ABRS18}, we don't lose any essential information when passing to kernels. Therefore, our constraining pencils are dual to the constraining pencils in \cite{ABRS18}. We modified the definition to fit into our framework of Schubert varieties. See \Cref{def:Schubert_variety} and \Cref{thm:pencil_equal_Schubert} for more details. \\

	To any $\Psi\in M_{m\times n}(\R)$, we attach an $m$-dimensional subspace $V_{\Psi}\subset\R^{m+n}$ which is spanned by the row vectors of the full rank $m\times(m+n)$ matrix
	\begin{equation}
	\begin{bmatrix}
	I_{m\times m}\vert \Psi
	\end{bmatrix}.
	\end{equation} 
	\par 
	Let $\varphi\colon[a,b]\rightarrow M_{m\times n}(\R)$ be an analytic curve. It induces a curve on $\mathrm{Gr}(m,m+n)$ by
	\begin{align*}
	\Phi\colon[a,b]&\longrightarrow \mathrm{Gr}(m,m+n)\\
	s &\longmapsto V_{\varphi(s)}.
	\end{align*}
	
	We identify $\mathrm{Gr}(m,m+n)$ with $G/P$, where $G=\SL_{m+n}(\R)$ and $P=P(a)$ is the parabolic subgroup associated with $a(t)=\mathrm{diag}(t^n,\cdots,t^n,t^{-m},\cdots,t^{-m})$. Hence it makes sense to talk about partial flag subvarieties of $\mathrm{Gr}(m,m+n)$. (See \Cref{def:partial_flag_subvariety}.) \par 
	Now we are ready for our main theorem on DT-improvability.
	\begin{theorem}[DT-improvability]\label{thm:Dirichlet}
		Let $\varphi\colon[a,b]\rightarrow M_{m\times n}(\R)$ be an analytic curve. Suppose that both of the following hold:
		\begin{enumerate}[(A)]
			\item The image of $\Phi$ is not contained in any weakly constraining pencil;
			\item The image of $\Phi$ is not contained in any proper partial flag subvariety of the Grassmannian variety $\mathrm{Gr}(m,m+n)$ with respect to $a(t)$.
		\end{enumerate}
		Then for Lebesgue almost every $s\in [a,b]$, $\varphi(s)$ is not DT-improvable.
	\end{theorem}
	\Cref{thm:Dirichlet} follows from \Cref{thm:equidistribution1} and \Cref{thm:pencil_equal_Schubert} via Dani's correspondence, as explained in \cite{KW08}\cite{Sha09Invention}\cite{Yan16Proc}\cite{SY16}. The proof also shows that for Lebesgue almost every $s\in [a,b]$, $\varphi(s)$ is not DT-improvable along $\mathcal{N}$ (see \cite{Sha10}), where $\mathcal{N}$ is any infinite set of positive integers.

	\subsection{Organization of the paper}
	In \Cref{sect:linear_stability}, we review the concept of Kempf's one-parameter subgroup, and use Kempf's numerical criterion to prove linear stability. \par 
	In \Cref{sect:nondivergence}, we review the $(C,\alpha)$-good property defined by Kleinbock and Margulis, and apply linearization technique combined with linear stability to prove non-divergence of translated measures. \par 
	In \Cref{sect:unipotent_invariance}, we apply the idea of twisting due to Shah, and prove a general result on unipotent invariance. \par 
	In \Cref{sect:focusing}, we use Ratner's theorem on unipotent flows and Dani-Margulis linearization technique to study the dynamical behavior of trajectories near singular sets, and obtain equidistribution results.\par 
	In \Cref{sect:Grassmannian_Schubert}, we study the special case of Grassmannians, and use Young diagrams to give a combinatorial description of constraining and weakly constraining pencils. \par 
	
	\subsection*{Acknowledgments}
	I would like to express my deep gratitude to my advisor Nimish Shah for suggesting this problem, for generously sharing his ideas, and for numerous helpful discussions and constant encouragement. I would like to thank Manfred Einsiedler and Alex Eskin for drawing \cite{Kha15} and \cite{KLP18} to my attention. I would like to thank David Anderson for many helpful discussions and his course on equivariant cohomology, where I learned a lot about Schubert varieties. Thanks are due to Menny Aka, Jayadev Athreya, Asaf Katz, Shi Wang, Barak Weiss and Runlin Zhang for helpful discussions and suggestions.\par
	Special thanks to Osama Khalil and Dmitry Kleinbock for discussions which helped me to find out a mistake in an earlier version of this paper. \par 
	I would like to thank my wife, Yushu Hu, for her unconditional support.

	\section{Linear stability and Kempf's one-parameter subgroups}\label{sect:linear_stability}
	Let $G=\mathbf{G}(\R)$ be a semisimple connected real algebraic group. If $\delta\colon \mathbb{G}_m \rightarrow \mathbf{G}$ is a homomorphism of real algebraic groups, we call $\delta$ a \emph{multiplicative one-parameter subgroup} of $G$. We associate a parabolic subgroup with $\delta$ as:
	\begin{equation}
	P(\delta) := \{ g \in G \colon \lim_{t\to\infty} \delta(t)g\delta(t)^{-1} \text{ exists in $G$}\},
	\end{equation}
	Let $\Gamma(G)$ be the set of the multiplicative one-parameter subgroups of $G$. Following Kempf \cite{Kem78}, we define the \emph{Killing length} of a multiplicative one-parameter subgroup $\delta$ by the equation
	\begin{equation}
	2\| \delta \|^2 = \mathrm{Trace}[(\mathrm{ad}(\delta_*\mathrm{d}/\mathrm{d}t))^2],
	\end{equation}
	and it follows from the invariance of the Killing form that the Killing length is $G$-invariant. \par
	Now fix a multiplicative one-parameter subgroup $a$ of $G$. We choose and fix a maximal $\R$-split torus $T$ of $G$ containing $\{a(t)\}$. Let $\Gamma(T)$ be the set of the multiplicative one-parameter subgroups of $T$, and $X(T)$ be the set of characters of $T$. We define a pairing as following: if $\chi\in X(T)$ and $\delta\in\Gamma(T)$, $\langle \chi,\delta \rangle$ is the integer which occurs in the formula $\chi(\delta(t)) = t^{\langle \chi,\delta \rangle}$. Let $(\cdot,\cdot)$ denote the positive definite bilinear form on $\Gamma(T)$ such that $(\delta, \delta) = \| \delta \|^2$. \par 
	By a suitable choice of positive roots $R^+$, we may assume that $a$ is a dominant cocharacter of in $T$. Recall that the set $\Gamma^+(T)$ of dominant cocharacters of $T$ is defined by:
	\begin{equation}
	\Gamma^+(T) = \{ \delta \in \Gamma(T) \colon \langle \delta,\alpha \rangle \geq 0,\, \forall \alpha\in R^+ \}.
	\end{equation}
	Let $B$ be the corresponding minimal parabolic subgroup of $G$ whose Lie algebra consists of all the non-positive root spaces. 
	\par
	Let $P=P(a)$ be the parabolic subgroup associated with $a$. Let $W^P$ denote set of minimal length coset representatives of the quotient $W/W_{P}$, where $W = N_G(T)/Z_G(T)$ and $W_{P} = N_{P}(T)/Z_{P}(T)$ are Weyl groups of $G$ and $P$. Then $W$ acts on $\Gamma(T)$ by conjugation: $w\cdot\delta = w\delta w^{-1}$. Denote $\delta^w=w\cdot\delta$. We take the Bruhat order on $W^P$ such that $w' \leq w$ if and only if the closure of the Schubert cell $BwP$ contains $Bw'P$. We note that the Bruhat order coincides with the folding order defined in \cite{KLP18} (See \cite[Remark 3.8]{KLP18}). \par
	
	\begin{definition}[Schubert variety]\label{def:Schubert_variety}
		Given an element $w\in W^P$, the standard Schubert variety $X_w$ is the Zariski closure of the Schubert cell $BwP$. A \emph{Schubert variety} is a subvariety of $G/P$ of the form $gX_w$, where $g\in G$ and $w\in W^P$. \par 
		We say that a Schubert variety $gX_w$ is \textbf{unstable} with respect to $a(t)$ if there exists $\delta\in\Gamma^{+}(T)$ such that $(\delta, a^w) > 0$. We say that $gX_w$ is \textbf{weakly unstable} with respect to $a(t)$ if there exists non-trivial $\delta\in\Gamma^{+}(T)$ such that $(\delta, a^w) \geq 0$. \par 
		For short, we will just say unstable or weakly unstable Schubert variety if $a(t)$ is clear in the context.
	\end{definition} \par 
	
	\begin{remark}
		In this article, when we project from $G$ to $G/P$, we always take the following map
		\begin{align}
		\pi_P\colon&G\longrightarrow G/P \nonumber \\
		&g \longmapsto g^{-1}P. 
		\end{align}
		When we write $BwP$, we treat it as a subvariety of $G/P$; while $Pw^{-1}B$ is treated as a subset of $G$.
	\end{remark}
	
	For $\delta\in \Gamma^{+}(T)$, define the subset $W^+(\delta, a)$ of $W^P$ as
	\begin{equation}\label{eq:W_plus}
	W^+(\delta, a) = \{ w \in W^P \colon (\delta, a^w) > 0 \},
	\end{equation}
	and we define $W^{-}(\delta, a), W^{0+}(\delta, a)$ and $W^{0-}(\delta, a)$ similarly, with $<$, $\geq$ and $\leq$ in place of $>$ in \eqref{eq:W_plus} respectively. We note that $W^{+}(\delta,a)$ is a ``metric thickening'' as defined in \cite[Section 3.4]{KLP18}.
	
	\begin{lemma}\label{lem:union_of_Schubert}
		\begin{enumerate}[(a)]
			\item Let $w'\leq w$ be elements in $W^P$, and $\delta\in\Gamma^{+}(T)$. Then one has $(\delta, a^{w'})\geq (\delta, a^w)$.
			\item $\bigsqcup_{w\in W^+(\delta, a)}BwP$ is a finite union of unstable Schubert subvarieties of $G/P$.
			\item $\bigsqcup_{w\in W^{0+}(\delta, a)}BwP$ is a finite union of weakly unstable Schubert subvarieties of $G/P$.
		\end{enumerate}
	\end{lemma}
	
	\begin{proof}
		Both (b) and (c) follow from (a). For a proof of (a), see e.g. \cite[Lemma 3.4]{KLP18}.
	\end{proof}
	
	Let $\rho: G\rightarrow \GL(V)$ be any finite dimensional linear representation of $G$. Let us recall some notions from geometric invariant theory (see e.g. \cite{MFK94} for more details). A nonzero vector $v$ is called \emph{unstable} if the closure of the $G$-orbit $Gv$ contains the origin. $v$ is called \emph{semistable} if it is not unstable. For any $v\in V\backslash\{0\}$ and $\delta\in\Gamma(G)$, by \cite[Lemma 1.2]{Kem78} we can write $v=\sum v_i$ where $\delta(t)v_i = t^iv_i$. Define the numerical function $m(v,\delta)$ to be the maximal\footnote[1]{It is ``minimal'' in Kempf's original definition. Since we are taking limit as $t$ tends to $\infty$ instead of $0$, our numerical function is actually opposite to Kempf's.} $i$ such that $v_i\neq 0$.\par
	By a theorem of Kempf (see \cite[Theorem 4.2]{Kem78}), the function $m(v, \delta)/\|\delta\|$ has a negative minimum value $B_v$ on the set of non-trivial multiplicative one-parameter subgroups $\delta$. Let $\Lambda(v)$ denote the set of primitive multiplicative one-parameter subgroup $\delta$ such that $m(v, \delta) = B_v\cdot\|\delta\|$. Kempf \cite[Theorem 4.2]{Kem78} shows that the parabolic subgroup $P(\delta)$ does not depend on the choice of $\delta\in\Lambda(v)$, which is denoted by $P(x)$. Moreover, $\Lambda(v)$ is a principal homogeneous space under conjugation by the unipotent radical of $P(x)$. In particular, for any $\delta$ in $\Lambda(v)$ and $b$ in $P(x)$, we know that $b\delta b^{-1}$ is also contained in $\Lambda(v)$. \\
	
	For $v\in V\backslash\{0\}$, define
	\begin{equation}\label{eq:definiton_of_zero_unstable}
	G(v, V^{-}(a)) = \{ g\in G \colon gv \in V^{-}(a) \},
	\end{equation}
	where
	\begin{equation}
	V^{-}(a) = \{ v\in V \colon \lim_{t\to\infty}a(t)v = 0 \}.
	\end{equation}
	As noted in \cite[Section 3.3]{Kha15}, though the limits in \cite{Kem78} are defined algebraically, they coincide with limits in the Hausdorff topology induced from the usual topology on $\R$, by \cite[Lemma 1.2]{Kem78}. \par 
	Now we proceed to the main result of this section.
	
	\begin{proposition}\label{prop:basic_lemma}
		For any $v\in V\backslash\{0\}$, there exits $\delta_0\in\Gamma^{+}(T)$ and $g_0\in G$ such that
		\begin{equation}\label{eq:basic_inclusion}
		G(v,V^{-}(a)) \subset \bigsqcup_{w\in W^+(\delta_0, a)}Pw^{-1}Bg_0^{-1}.
		\end{equation}
	\end{proposition}
	\begin{proof}
		By definition we have the following identities because of $G$-equivariance:
		
		\begin{equation}
		G(gv, V^{-}(a)) = G(v, V^{-}(a))g^{-1},\quad \forall g\in G;
		\end{equation}
		\begin{equation}\label{eq:G_equivarianc_Lambda}
		\Lambda(gv) = g\Lambda(v)g^{-1},\quad \forall g\in G.
		\end{equation} \par 
		If $v$ is semistable, then $G(v, V^{-}(a))$ is empty, and the conclusion trivially holds. From now on we assume that $v$ is unstable, and thus $\Lambda(v)$ is non-empty. Take $\delta_1\in\Lambda(v)$, then there exists $g_0\in G$ and $\delta_0\in\Gamma^{+}(T)$ such that $g_0^{-1}\delta_1g_0 = \delta_0$. It follows from \eqref{eq:G_equivarianc_Lambda} that $\delta_0\in\Lambda(g_0^{-1}v)$.\par
		We prove by contradiction. Suppose that \eqref{eq:basic_inclusion} does not hold. Considering the Bruhat decomposition
		\begin{equation}
		G = \bigsqcup_{w\in W^P} Pw^{-1}B,
		\end{equation}
		we can take $g\in G(g_0^{-1}v, V^{-}(a))$ such that it can be written as
		\begin{equation}
		g = pw^{-1}b, \text{ where } p\in P,\, w\in W^{0-}(\delta_0, a),\, b\in B.
		\end{equation}
		Write $v' = bg_0^{-1}v$. In view of \eqref{eq:G_equivarianc_Lambda}, by \cite[Theorem 4.2(3)]{Kem78} we have $\Lambda(g_0^{-1}v)=\Lambda(v')$. Hence $\delta_0$ is an element in $\Lambda(v')$. \par 
		We also have $v'\in V^{-}(a^w)$. Indeed, $gg_0^{-1}v\in V^-(a)$ implies that $pw^{-1}v'\in V^{-}(a)$. Since $V^{-}(a)$ is $P$-invariant, we know that $w^{-1}v'\in V^{-}(a)$. Hence $v'\in V^{-}(a^w)$. \par
		Take a large integer $N$, we define $\delta_{N} = N\delta_0 + a^w$. We claim that for a sufficiently large $N$, one has 
		\begin{equation}\label{eq:destablize_faster}
		\frac{m(v', \delta_{N})}{\| \delta_{N} \|} < \frac{m(v', \delta_0)}{\| \delta_0 \|},
		\end{equation}
		and this will contradict the fact that $\delta_0\in\Lambda(v')$.\par 
		To prove the claim, consider the weight space decomposition $V = \bigoplus V_{\chi}$, where $T$ acts on $V_{\chi}$ by multiplication via the character $\chi$ of $T$. It suffices to prove that for any $\chi$ such that the projection of $v'$ on $V_{\chi}$ is nonzero, one has
		\begin{equation}\label{eq:higher_speed}
		\frac{\langle \chi, \delta_{N} \rangle}{\|\delta_{N} \|} <
		\frac{\langle \chi, \delta_0 \rangle}{\| \delta_0 \|}.
		\end{equation}
		To prove \eqref{eq:higher_speed}, we define an auxiliary function:
		\begin{equation}
		\begin{split}
		f(s) &= \frac{\langle \chi, \delta_0 + s\cdot a^w \rangle ^ 2}{\| \delta_0 + s \cdot a^w \| ^ 2}\\
		&= \frac{\langle \chi , \delta_0 \rangle ^ 2 + 2s\langle \chi, \delta_0 \rangle\langle \chi, a^w \rangle + s^2 \langle \chi, a^w \rangle ^2}
		{(\delta_0, \delta_0) + 2s(\delta_0, a^w) + s^2(a^w, a^w)}
		\end{split}
		\end{equation}
		Compute its derivative at 0:
		\begin{equation}
		f'(0) = \frac{2\langle \chi,\delta_0\rangle\langle\chi,a^w\rangle(\delta_0, \delta_0) - 2(\delta_0, a^w)\langle\chi, \delta_0\rangle^2}
		{(\delta_0, \delta_0)^2}
		\end{equation}
		Since $v'\in V^{-}(a^w)$, we know that $\langle \chi, a^w \rangle < 0$. Since $\delta_0\in\Lambda(v')$, we know that $\langle \chi, \delta_0 \rangle < 0$. Also by the choice of $w$ we know that $(\delta_0, a^w) \leq 0$. Combining the above one gets $f'(0) > 0$. Hence for $N$ large we have
		\begin{equation}\label{eq:temp14}
		f(1/N) > f(0),
		\end{equation}
		and \eqref{eq:higher_speed} follows because each side of \eqref{eq:temp14} is the square of each side of \eqref{eq:higher_speed}. Therefore \eqref{eq:destablize_faster} holds, contradicting the fact that $\delta_0\in\Lambda(v')$.
	\end{proof}
	
	Now we are ready to prove \Cref{thm:linear_stability}.
	
	\begin{proof}[Proof of \Cref{thm:linear_stability}]
		We prove by contradiction. Suppose that for all $C>0$, there exist $t$ and $v$ such that \eqref{eq:linear_stablility} does not hold. We take a sequence $C_i\to 0$. Then after passing to a subsequence we can find $t_i\to\infty$ and a sequence $(v_i)_{i\in\N}$ in $V$ such that
		\begin{equation}
		\sup_{s\in [a,b]}\lVert a(t_i)\phi(s)v_i \rVert < C_i\lVert v_i \rVert.
		\end{equation}
		Without loss of generality we may assume that $\lVert v_i \rVert = 1$. Then after passing to a subsequence, we may assume that $v_i \to v_0$. Hence we have
		\begin{equation}
		\sup_{s\in [a,b]}\lVert a(t_i)\phi(s)v_0 \rVert \stackrel{t_i\to\infty}{\longrightarrow} 0.
		\end{equation}
		Therefore $\phi(s)v_0$ is contained in $V^{-}(a)$ for all $s\in[a,b]$, and it follows that the image of $\phi$ is contained in $G(v_0, V^{-}(a))$. (See \eqref{eq:definiton_of_zero_unstable}.) By \Cref{lem:union_of_Schubert}(b) and \Cref{prop:basic_lemma}, the image of $G(v_0, V^{-}(a))$ under $\pi_P$ in $G/P$ is a finite union of unstable Schubert varieties. But $\phi$ is analytic, which implies that the image of $\widetilde{\phi}$ is contained in one single unstable Schubert variety. This contradict our assumption on $\phi$.
	\end{proof}
	
	\Cref{prop:basic_lemma} and \Cref{thm:linear_stability} will play a central role in proving the non-divergence of translated measures. To handle non-focusing, one needs a slightly generalized version, motivated by the work of Richard and Shah \cite[Section 6]{RS16}. We need the following result due to Kempf.
	
	\begin{lemma}[\cite{Kem78} Lemma 1.1(b)]\label{lem:Kempf_reduction}
		Let $G$ be a connected reductive algebraic group over a field $k$, and $X$ be any affine $G$-scheme. If $S$ is a closed $G$-subscheme of $X$, then there is a $G$-equivariant morphism $f:X\rightarrow W$, where $W$ is a representation of $G$, such that $S$ is the scheme-theoretic inverse image $f^{-1}(0)$ of the reduced closed subscheme of $W$ supported by zero.
	\end{lemma} \par 
	
	In view of Kempf's \Cref{lem:Kempf_reduction}, the following is a corollary of \Cref{prop:basic_lemma}.
	
	\begin{corollary}\label{cor:basic_lemma1}
		Let the notations be as in the beginning of this section. Let $S$ be the real points of any $G$-subscheme of $V$. For any $v\in V$, define the following subset of $G$:
		\begin{equation}\label{eq:cor_basic_lemma1}
		G(v, S,a) = \{ g\in G \colon \lim_{t\to\infty}a(t)gv\in S \}.
		\end{equation}
		Then for any $v\in V\backslash S$, there exists $\delta_0\in\Gamma^{+}(T)$ and $g_0\in G$ such that
		\begin{equation}
		G(v, S,a) \subset \bigsqcup_{w\in W^+(\delta_0, a)}Pw^{-1}Bg_0^{-1}.
		\end{equation}
	\end{corollary}
	
	\begin{proof}
		By \Cref{lem:Kempf_reduction}, there exist a $G$-equivariant morphism $f:V\rightarrow W$ where $f^{-1}(0) = S$. Hence it follows from the definition that
		\begin{equation}
		G(v, S, a) \subset G(f(v), W^{-}(a)).
		\end{equation}
		Now it remains to apply \Cref{prop:basic_lemma} for $W$ and $f(v)$.
	\end{proof}
	
	Now we present the following variance of \Cref{prop:basic_lemma}.
	
	\begin{proposition}\label{prop:basic_lemma2}
		Let $v\in V$ such that the $G$-orbit $Gv$ is not closed. Define
		\begin{equation}
		G(v, V^{0-}(a)) = \{ g\in G : gv\in V^{0-}(a) \},
		\end{equation}
		where
		\begin{equation}
		V^{0-}(a) = \{ v\in V : \lim_{t\to \infty} a(t)v \text{ exists} \}.
		\end{equation}
		Then there exists $\delta_0\in\Gamma^{+}(T)$ and $g_0\in G$ such that
		\begin{equation}\label{eq:basic_lemma2}
		G(v, V^{0-}(a)) \subset \bigsqcup_{w\in W^{0+}(\delta_0, a)}Pw^{-1}Bg_0^{-1}.
		\end{equation}
	\end{proposition}
	
	\begin{proof}
		Let $S=\partial(Gv)$. Since any $G$-orbit is open in its closure, we know that $S$ is closed and $G$-invariant. By \Cref{lem:Kempf_reduction}, there exists a $G$-equivariant morphism $f:V\rightarrow W$ where $f^{-1}(0) = S$. Notice that $f(v)$ is unstable in $W$. We claim that
		\begin{equation}\label{eq:temp10}
		G(f(v),W^{0-}(a))\subset \bigsqcup_{w\in W^{0+}(\delta_0, a)}Pw^{-1}Bg_0^{-1}.
		\end{equation}
		\par 
		To prove the claim, we argue with $W$ and $f(v)$ in exactly the same way as in the proof of \Cref{prop:basic_lemma}. The only difference is the following. When showing $f'(0)>0$, one needs $\langle\chi, a^w\rangle < 0$ and $(\delta_0, a^w)\leq 0$ there; but here one has $\langle\chi, a^w\rangle \leq 0$ and $(\delta_0, a^w)< 0$, which also implies that $f'(0)>0$. Hence \eqref{eq:temp10} holds. \par 
		Finally, since $f$ is $G$-equivariant, we have $f(V^{0-})\subset W^{0-}$. Hence
		\begin{equation}
		G(v, V^{0-}(a)) \subset G(f(v),W^{0-}(a)).
		\end{equation}
		Therefore \eqref{eq:basic_lemma2} holds.
	\end{proof}

	\section{Non-divergence of the limiting distribution}\label{sect:nondivergence}
	Let $G=\mathbf{G}(\R)$ be a connected semisimple real algebraic group, and $L$ be a real Lie group containing $G$. Let $\{a(t)\}_{t\in\R^{\times}}$ be a multiplicative one-parameter subgroup of $G$ with non-trivial projection on each simple factor of $G$. Let $P=P(a)$ be the parabolic subgroup of $G$ whose real points consists of the elements $g\in G$ such that the limit $\lim_{t\to \infty}a(t)ga(t)^{-1}$ exists. Let $\phi \colon I = [a,b] \rightarrow G$ be an analytic map, and let $\pi_P \colon G \rightarrow G/P$ be the projection which maps $g$ to $g^{-1}P$. Then $\widetilde{\phi} = \pi_P\circ\phi$ is an analytic curve on $G/P$. In this section we assume that the image of $\widetilde{\phi}$ is not contained in any unstable Schubert variety of $G/P$ with respect to $a(t)$. \par
	Let $x_0 = l\Lambda \in L/\Lambda$. We will assume that the orbit of $x_0$ under $G$ is dense in $L/\Lambda$; that is $\overline{Gx_0} = L/\Lambda$. Let $t_i\to \infty$ be any sequence in $\R_{>0}$. Let $\mu_i$ be the parametric measure supported on $a(t_i)\phi(I)x_0$, that is, for any compactly supported function $f \in C_c(L/\Lambda)$ one has
	\begin{equation}\label{eq:mu_i}
	\int_{L/\Lambda}f\,d\mu_i = \frac{1}{|I|}\int_{I}f(a(t_i)\phi(s)x_0)\,\mathrm{d}s.
	\end{equation}
	
	\begin{theorem}\label{thm:non_divergence}
		Given $\epsilon > 0$ there exists a compact set $\mathcal{F} \subset L/\Lambda$ such that $\mu_i(\mathcal{F}) \geq 1-\epsilon$ for all large $i\in\N$.
	\end{theorem}
	
	This theorem will be proved via linearization technique combined with \Cref{thm:linear_stability}. We follow \cite[Section 3]{Sha09Invention} closely, as most of the arguments there work not only for $G=\SL_n(\R)$ but also for general $G$.
	
	\begin{definition}\label{def:V}
		Let $\mathfrak{l}$ denote the Lie algebra of $L$, and denote $d = \mathrm{dim}\,L$. We define
		\begin{equation*}
		V = \bigoplus_{i=1}^d \bigwedge\nolimits^{\!i} \mathfrak{l},
		\end{equation*}
		and let $L$ act on $V$ via $\bigoplus_{i=1}^{d}\bigwedge\nolimits^{\!i}\mathrm{Ad}(L)$. This defines a linear representation of $L$ (and of $G$ by restriction):
		\begin{equation*}
		L \rightarrow \GL(V).
		\end{equation*}
	\end{definition}
	
	The following theorem due to Kleinbock and Margulis is the basic tool to prove that there is no escape of mass to infinity:
	
	\begin{theorem}[see \cite{Dan84}, \cite{KM98} and \cite{Sha09Duke1}]\label{thm:nondivergence_criterion}
		Fix a norm $\lVert\cdot\rVert$ on $V$. There exist finitely many vectors $v_1, v_2, \cdots, v_r \in V$ such that for each $i = 1, 2, \cdots, r$, the orbit $\Lambda v_i$ is discrete, and moreover, the following holds: for any $\epsilon > 0$ and $R > 0$, there exists a compact set $\mathcal{F} \subset L/\Lambda$ such that for any $t > 0$ and any subinterval $J \subset I$, one of the following holds:
		\begin{enumerate}[(I)]
			\item There exist $\gamma \in \Lambda$ and $j \in \{ 1, \cdots, r \}$ such that
			\begin{equation*}
			\sup_{s\in J}\lVert a(t)\phi(s)l\gamma v_j\rVert < R;
			\end{equation*}
			\item
			\begin{equation*}
			\lvert \{ s\in J \colon a(t)\phi(s)x_0 \in K \} \rvert \geq (1-\epsilon)\lvert J \rvert.
			\end{equation*}
		\end{enumerate}
	\end{theorem}
	
	The key ingredient of the proof, as explained in \cite[Section 3.2]{Sha09Invention} and \cite[Section 2.1]{Sha09Duke1}, is the following growth property called the $(C, \alpha)$-good property, which is due to \cite[Proposition 3.4]{KM98}. Following Kleinbock and Margulis, we say that a function $f\colon I \rightarrow \R$ is $(C, \alpha)$-good if for any subinterval $J \subset I$ and any $\epsilon > 0$, the following holds:
	\begin{equation*}
	\lvert \{ s\in J\colon \lvert f(s) \rvert < \epsilon \} \rvert 
	\leq C\left(\frac{\epsilon}{\sup_{s\in J}\lvert f(s) \rvert}\right)^{\alpha}\lvert J \rvert.
	\end{equation*} \\
	
	Now we are ready to prove the main result of this section.
	
	\begin{proof}[Proof of \Cref{thm:non_divergence}]
		Take any $\epsilon > 0$. Take a sequence $R_k \to 0$ as $k \to \infty$. For each $k \in \N$, let $\mathcal{F}_k \subset L/\Lambda$ be a compact set as determined by \Cref{thm:nondivergence_criterion} for these $\epsilon$ and $R_k$. If the theorem fails to hold, then for each $k \in \N$ we have $\mu_i(\mathcal{F}_k) > 1 - \epsilon$ for infinitely many $i \in \N$. Therefore after passing to a subsequence of $\{ \mu_i \}$, we may assume that $\mu_i(\mathcal{F}_i) < 1-\epsilon$ for all $i$. Then by \Cref{thm:nondivergence_criterion}, after passing to a subsequence, we may assume that there exists $v_0$ and $\gamma_i\in\Lambda$ such that
		\begin{equation*}
		\sup_{s\in I} \lVert a(t_i)\phi(s)l\gamma_iv_0 \rVert \leq R_i\stackrel{i\to\infty}{\longrightarrow} 0.
		\end{equation*}
		Since $\Lambda\cdot v_0$ is discrete, there exists $r_0 > 0$ such that $\lVert l\gamma_iv_0 \rVert \geq r_0$ for each i. We put $v_i = l\gamma_iv_0/\lVert l\gamma_iv_0 \rVert$. Then $v_i \rightarrow v \in V$ and $\lVert v \rVert = 1$. Therefore
		\begin{equation}
		\sup_{s\in I} \lVert a(t_i)\phi(s)v_i \rVert \leq R_i/r_0\stackrel{i\to\infty}{\longrightarrow} 0.
		\end{equation}
		Then it follows that
		\begin{equation}\label{eq:uniform_shrinking}
		\sup_{s\in I} \lVert a(t_i)\phi(s)v \rVert \stackrel{i\to\infty}{\longrightarrow} 0.
		\end{equation}
		This contradict \Cref{thm:linear_stability}.
	\end{proof}
	
	As a consequence of \Cref{thm:non_divergence}, we deduce the following:
	\begin{corollary}\label{cor:weak_limit_exists}
		After passing to a subsequence, $\mu_i\to\mu$ in the space of probability measures on $L/\Lambda$ with respect to the weak-* topology.
	\end{corollary} \par 
	We note that \Cref{thm:non_divergence_introduction} follows from \Cref{thm:non_divergence}.

	\section{Invariance under a unipotent flow}\label{sect:unipotent_invariance}
	Let $G = \mathbf{G}(\R)$ be a connected semisimple real algebraic group, and $\{a(t)\}_{t\in\R^{\times}}$ be a multiplicative one-parameter subgroup of $G$ with non-trivial projection on each simple factor of $G$. Define
	\begin{equation}
	P = \{g\in G \colon \lim_{t\to\infty}a(t)ga(t)^{-1}\text{ exists}\}.
	\end{equation}
	Let $\mathfrak{X}$ be a locally compact second countable Hausdorff topological space, with a continuous $G$-action. Let $\phi\colon I = [a, b]\rightarrow G$ be an analytic curve, whose projection under $g\mapsto g^{-1}P$ on $G/P$ is non-trivial. Let $\mathfrak{g}$ denote the Lie algebra of $G$. \par 
	Since the exponential map $\exp\colon\mathfrak{g}\rightarrow G$ is a local homeomorphism, we can take a sufficiently small $\eta>0$ such that for any $s\in I$ and $0<\xi<\eta$, there exists $\Psi(s,\xi)$ in $\mathfrak{g}$ such that
	\begin{equation}\label{eq:def_Psi}
	\phi(s+\xi)\phi(s)^{-1}= \exp \Psi(s,\xi).
	\end{equation}
	Moreover, $\Psi$ is an analytic map in both $s$ and $\xi$.
	
	\begin{lemma}\label{lem:converge_to_one_direction}
		There exists $m>0$ and a nilpotent element $Y_s$ in $\mathfrak{g}$ such that for all but finitely many $s\in I$,
		\begin{equation}\label{eq:converge_to_one_direction}
		\mathrm{Ad}\,a(t)\,\Psi(s,t^{-m})\rightarrow Y_s,\quad t\to\infty.
		\end{equation}
		Moreover, one can assume that the map $s\rightarrow Y_s$ is analytic, and the convergence is uniform in $s$.
	\end{lemma}
	
	\begin{proof}
		Since $\Psi$ is an analytic map in both $s$ and $\xi$, we can write
		\begin{equation}\label{eq:temp11}
		\Psi(s,\xi) = \sum_{i=1}^{\infty}\xi^i\psi_i(s),
		\end{equation}
		where $\psi_i\colon I\rightarrow\mathfrak{g}$ is analytic for each $i$. \par 
		Notice that $\mathrm{Ad}\,a(t)$ is semisimple and acts on the finite dimensional vector space $\mathfrak{g}$, then for each $i$ there exist $m_i\in\Z$ such that
		\begin{equation}\label{eq:temp12}
		\mathrm{Ad}\,a(t)\psi_i(s) = \sum_{j\leq m_i}t^j \psi_{i,j}(s),
		\end{equation}
		where $\psi_{i,j}(s)$ is analytic in $s$, and $\psi_{i,m_i}(s)\neq 0$ for all but finitely many $s\in I$. Since the projection of $\phi$ on $G/P$ is non-trivial, there exists $i$ such that $m_i>0$.\par 
		Combining \eqref{eq:temp11}\eqref{eq:temp12}, we get
		\begin{equation}\label{eq:temp13}
		\mathrm{Ad}\,a(t)\,	\Psi(s,\xi) = \sum_{i=0}^{\infty}\sum_{j\leq m_i}t^j\xi^i\psi_{i,j}(s).
		\end{equation} \par 
		Now set $m=\max_{i\geq 1} \{m_i/i\}$. Since $m_i$ are all eigenvalues of $\mathrm{Ad}\,a(t)$, they are uniformly bounded from above. Hence we know that $m$ exists and $m>0$. Denote $I=\{i\geq 1\colon m_i/i = m \}$, and we see that $I$ is a finite set. We set
		\begin{equation}
		Y_s = \sum_{i\in I}\psi_{i,m_i}(s).
		\end{equation}
		Since the eigenvalues of $\mathrm{Ad}\,a(t)$ acting on $Y_s$ are all positive, $Y_s$ is nilpotent. \par 
		In view of \eqref{eq:temp13},
		\begin{equation}
		\mathrm{Ad}\,a(t)\,	\Psi(s,t^{-m}) = Y_s + \sum_{j-im<0}t^{j-im}\psi_{i,j}(s),
		\end{equation}
		and \eqref{eq:converge_to_one_direction} follows.
	\end{proof}
	
	We could then twist $Y_s$ into one direction due to the following lemma.
	
	\begin{lemma}
		There are only finitely many $G$-conjugacy classes of the nilpotent elements in the Lie algebra $\mathfrak{g}$ of $G$.
	\end{lemma}
	
	\begin{proof}
		This result has been proved for groups over the complex numbers $\C$ (see \cite{Ric67}). Let $X$ be any non-zero nilpotent element in $\mathfrak{g}$. Now it remains to show that there are only finitely many $\mathbf{G}(\R)$-orbits in the real points of  $\mathbf{G}(\C)\cdot X$. Let $\mathbf{H}$ be the stabilizer of $X$ in $\mathbf{G}$. Then $\mathbf{H}$ is an algebraic group defined over $\R$. It is well known that the $\mathbf{G}(\R)$-orbits in $(\mathbf{G}/\mathbf{H})(\R)$ are parametrized by the Galois cohomology $H^1(\mathrm{Gal}(\C/\R),\mathbf{H}(\C))$. Then the statement of the lemma follows from the finiteness of $H^1(\mathrm{Gal}(\C/\R),\mathbf{H}(\C))$, which is guaranteed by \cite[Theorem 6.14]{PR94}.
	\end{proof}
	
	Since there are only finitely many conjugacy classes of nilpotent elements in $\mathfrak{g}$, up to at most finitely many points we may assume that all the $Y_s$ are in the same conjugacy class. Hence there exists $w_0$ in $\mathfrak{g}$, and $\delta(s)$ in $G$ which is also analytic in $s$, such that for all but finitely many $s\in I$ one has
	\begin{equation}
	\mathrm{Ad}(\delta(s))\cdot Y_s = w_0.
	\end{equation}
	Define the unipotent one-parameter subgroup of $G$ as
	\begin{equation}
	W = \{\exp(tw_0)\colon t\in\R\}.
	\end{equation}
	\par 
	Let $(t_i)_{i\in\N}$ be a sequence in $\R$ such that $t_i\to\infty$ as $i\to\infty$. Let $x_i\to x$ a convergent sequence in $\mathfrak{X}$. For each $i\in\N$, let $\lambda_i$ be the probability measure on $\mathfrak{X}$ such that
	\begin{equation}\label{eq:def_lambda_i}
	\int_{\mathfrak{X}} f\,\mathrm{d}\lambda_i = \frac{1}{\lvert I \rvert}\int_{s\in I}f(\delta(s) a(t_i)\phi(s)x_i)\,\mathrm{d}s,\quad\forall f\in \mathrm{C}_c(\mathfrak{X}).
	\end{equation}
	
	The following theorem is the main result of this section. The new idea here due to Nimish Shah is that we can actually twist the curve after translating by $a(t)$.
	
	\begin{theorem}\label{thm:unipotent_invariance}
		Suppose that $\lambda_i\to\lambda$ in the space of finite measures on $\mathfrak{X}$ with respect to the weak-* topology, then $\lambda$ is invariant under $W$.
	\end{theorem}
	
	\begin{proof}
		Given $f\in\mathrm{C}_c(\mathfrak{X})$ and $\epsilon > 0$. Since $f$ is uniformly continuous, there exists a neighborhood $\Omega$ of the neutral element in $G$ such that
		\begin{equation}\label{eq:uniform_continuity}
		\lvert f(\omega y) - f(y) \rvert < \epsilon,\quad\forall \omega\in\Omega,\,\forall y\in \mathfrak{X}.
		\end{equation}
		Define
		\begin{equation}
		\Omega' = \bigcap_{s\in I} \delta(s)^{-1}\Omega\delta(s),
		\end{equation}
		and $\Omega'$ is non-empty and open because $\{\delta(s)\}_{s\in I}$ is compact.\par
		By \Cref{lem:converge_to_one_direction}, there exists $T>0$ such that for all $t>T$ and for all but finitely many $s\in I$, there exists $\omega_{t,s}\in\Omega'$ such that
		\begin{equation}
		a(t)\exp\Psi(s,t^{-m})a(t)^{-1} = \omega_{t,s} \exp Y_s.
		\end{equation}
		Take $\xi_i=t_i^{-m}$. In view of \eqref{eq:def_Psi}, for $i$ large enough we have
		\begin{equation}
		\phi(s+\xi_i) = \exp\Psi(s,\xi_i)\phi(s).
		\end{equation}
		Hence there exists $i_0\in\N$ such that for all $i>i_0$,
		\begin{equation}
		\begin{split}
		\delta(s)a(t_i)\phi(s+\xi_i) &= \delta(s)a(t_i)\exp\Psi(s,\xi_i)\phi(s)\\
		&= \delta(s)\omega_{t_i,s}\exp Y_sa(t_i)\phi(s)\\
		&=\left(\delta(s)\omega_{t_i,s}\delta(s)^{-1}\right)\delta(s)\exp Y_sa(t_i)\phi(s)\\
		&=\left(\delta(s)\omega_{t_i,s}\delta(s)^{-1}\right)(\exp w_0)\delta(s)a(t_i)\phi(s)\\
		&\in \Omega(\exp w_0)\delta(s)a(t_i)\phi(s).
		\end{split}
		\end{equation}
		By \eqref{eq:uniform_continuity} we know that for all but finitely many $s\in I$,
		\begin{equation}
		\lvert f((\exp w_0)\delta(s)a(t_i)\phi(s)x_i) - f(\delta(s)a(t_i)\phi(s+\xi_i)x_i) \rvert < \epsilon.
		\end{equation}
		It follows that for all $i>i_0$,
		\begin{equation}
		\left\vert \frac{1}{\lvert I \rvert}\int_{I}f((\exp w_0)\delta(s)a(t_i)\phi(s)x_i)\,\mathrm{d}s - \frac{1}{\lvert I \rvert}\int_{I}f(\delta(s)a(t_i)\phi(s+\xi_i)x_i)\,\mathrm{d}s \right\vert < \epsilon.
		\end{equation}
		On the other hand, since $f$ is bounded on $\mathfrak{X}$, there exists $i_1\in\N$ such that for all $i > i_1$,
		\begin{equation}
		\left\vert\frac{1}{\lvert I \rvert}\int_{I}f(\delta(s)a(t_i)\phi(s+\xi_i)x_i)\,\mathrm{d}s - \frac{1}{\lvert I \rvert}\int_{I}f(\delta(s)a(t_i)\phi(s)x_i)\,\mathrm{d}s \right\vert < \epsilon.
		\end{equation}
		Combining the above two equations we get
		\begin{equation}
		\left\vert \frac{1}{\lvert I \rvert}\int_{I}f((\exp w_0)\delta(s)a(t_i)\phi(s)x_i)\,\mathrm{d}s - \frac{1}{\lvert I \rvert}\int_{I}f(\delta(s)a(t_i)\phi(s)x_i)\,\mathrm{d}s \right\vert < 2\epsilon.
		\end{equation}
		Therefore, for $i$ large enough we have
		\begin{equation}
		\left\vert\int_{\mathfrak{X}} f((\exp w_0)\cdot x)\,\mathrm{d}\lambda_i - \int_{\mathfrak{X}} f(x)\,\mathrm{d}\lambda_i \right\vert < 2\epsilon.
		\end{equation}
		Taking $i\to\infty$,
		\begin{equation}
		\left\vert\int_{\mathfrak{X}} f((\exp w_0)\cdot x)\,\mathrm{d}\lambda - \int_{\mathfrak{X}} f(x)\,\mathrm{d}\lambda \right\vert \leq 2\epsilon.
		\end{equation}
		Since $\epsilon$ is arbitrary, we conclude that $\lambda$ is $\exp w_0$-invariant. \par 
		If we replace $w_0$ with any scalar multiple of $w_0$, the above arguments still work. Hence $\lambda$ is invariant under $W = \{\exp(tw_0)\colon t\in\R\}$.
	\end{proof}

	\section{Dynamical behavior of translated trajectories near singular sets}\label{sect:focusing}
	Let notations be as in \Cref{sect:nondivergence}. Recall that the image of $\widetilde{\phi}$ is not contained in any unstable Schubert varieties of $G/P$ with respect to $a(t)$. Let $\{\lambda_i:i\in\N\}$ be the sequence of probability measures on $L/\Lambda$ as define in \eqref{eq:def_lambda_i}, where we take $\mathfrak{X} = L/\Lambda$ and $x_i=x_0$. Due to \Cref{thm:non_divergence}, by passing to a subsequence we assume that $\lambda_i\to\lambda$ as $i\to\infty$, where $\lambda$ is a probability measure on $L/\Lambda$. By \Cref{thm:unipotent_invariance}, $\lambda$ is invariant under a unipotent subgroup $W$. We would like to describe the limit measure $\lambda$ using the description of ergodic invariant measures for unipotent flows on homogeneous spaces due to Ratner \cite{Rat91}. We follow the treatment in \cite[Section 4]{Sha09Duke1}. \par 
	
	\subsection{Ratner's theorem and linearization technique}
	Let $\pi \colon L \rightarrow L/\Lambda$ denote the natural quotient map. Let $\mathcal{H}$ denote the collection of closed connected subgroups $H$ of $L$ such that $H\cap\Lambda$ is a lattice in $H$, and suppose that a unique unipotent one-parameter subgroup of $H$ acts ergodically with respect to the $H$-invariant probability measure on $H/H\cap\Lambda$. Then $\mathcal{H}$ is a countable collection (see \cite{Rat91}).\par 
	For a closed connected subgroup $H$ of $L$, define
	\begin{equation}
	N(H,W) = \{g\in L\colon g^{-1}Wg \subset H \}.
	\end{equation}
	Now, suppose that $H\in\mathcal{H}$. We define the associated singular set
	\begin{equation}
	S(H,W) = \bigcup_{\stackrel {F\in\mathcal{H}}{F\subsetneq H}}N(F,W).
	\end{equation}
	Note that $N(H,W)N_L(H)=N(H,W)$. By \cite[Proposition 2.1, Lemma 2.4]{MS95},
	\begin{equation}
	N(H,W)\cap N(H,W)\gamma \subset S(H,W), \; \forall \gamma\in\Lambda\backslash N_G(H).
	\end{equation}
	By Ratner's theorem \cite[Theorem 1]{Rat91}, as explained in \cite[Theorem 2.2]{MS95}, we have the following.
	\begin{theorem}[Ratner]
		Given the $W$-invariant probability measure $\lambda$ on $L/\Lambda$, there exists $H\in\mathcal{H}$ such that
		\begin{equation}\label{eq:positive_limit_measure_on_singular_set}
		\lambda(\pi(N(H,W)))>0 \quad \text{and} \quad \lambda(\pi(S(H,W)))=0.
		\end{equation}
		Moreover, almost every $W$-ergodic component of $\lambda$ on $\pi(N(H,W))$ is a measure of the form $g\mu_H$, where $g\in N(H,W)\backslash S(H,W)$ and $\mu_H$ is a finite $H$-invariant measure on $\pi(H)\cong H/H\cap\Lambda$. In particular if $H$ is a normal subgroup of $L$ then $\lambda$ is $H$-invariant.
	\end{theorem}\par 
	Let $V$ be as in \Cref{sect:nondivergence}. Let $d = \dim H$, and fix $p_H \in \bigwedge^d\mathfrak{h}\backslash \{0\}$. Due to \cite[Theorem 3.4]{DM93}, the orbit $\Lambda p_H$ is a discrete subset of $V$. We note that for any $g\in N_L(H)$, $gp_H = \det(\mathrm{Ad}\,g\vert_\mathfrak{h})p_H$. Hence the stabilizer of $p_H$ in $L$ equals
	\begin{equation}
	N_L^1(H) := \{ g\in N_L(H)\colon \det(\mathrm{Ad}\,g\vert_\mathfrak{h})=1\}.
	\end{equation}
	Recall that $\mathrm{Lie}(W)=\R w_0$. Let
	\begin{equation}
	\mathcal{A} = \{ v\in V \colon v\wedge w_0 = 0 \},
	\end{equation}
	where $V$ is defined in \Cref{def:V}. Then $\mathcal{A}$ is a linear subspace of $V$. We observe that
	\begin{equation}
	N(H,W) = \{ g\in L\colon g\cdot p_H \in \mathcal{A} \}.
	\end{equation}\par
	Recall that $x_0=l\Lambda\in L/\Lambda$. Using the fact that $\phi$ is analytic, we obtain the following consequence of the linearization technique and $(C,\alpha)$-good property (see \cite{Sha09Duke1}\cite{Sha09Invention}\cite{Sha10}).
	
	\begin{proposition}\label{prop:nonfocusing_creterion}
		Let $C$ be a compact subset of $N(H,W)\backslash S(H,W)$. Given $\epsilon > 0$, there exists a compact set $\mathcal{D}\subset\mathcal{A}$ such that, given a relatively compact neighborhood $\Phi$ of $\mathcal{D}$ in $V$, there exists a neighborhood $\mathcal{O}$ of $\pi(C)$ in $L/\Lambda$ such that for any $t\in\R$ and subinterval $J\subset I$, one of the following statements holds:
		\begin{enumerate}[(I)]
			\item $\lvert \{ s\in J \colon \delta(s)a(t)\phi(s)x_0 \in \mathcal{O} \} \rvert \leq \epsilon \lvert J \rvert$.
			\item There exists $\gamma\in\Lambda$ such that $\delta(s)a(t)\phi(s)l\gamma p_H \in \Phi$ for all $s\in J$.
		\end{enumerate}
	\end{proposition}\par
	
	\subsection{Algebraic consequences of positive limit measure on singular sets}
	Recall the definition of $\lambda_i$ in \eqref{eq:def_lambda_i}, where we take $\mathfrak{X} = L/\Lambda$ and $x_i=x_0$. After passing to a subsequence, $\lambda_i\to\lambda$ in the space of probability measures on $L/\Lambda$, and by \Cref{thm:non_divergence} and \Cref{thm:unipotent_invariance}, we know that there exists $H\in\mathcal{H}$ such that
	\begin{equation}
	\lambda(\pi(N(H,W)\backslash S(H,W))>0.
	\end{equation}\par 
	In this section, we use \Cref{prop:nonfocusing_creterion} and \Cref{thm:linear_stability} to obtain the following algebraic consequence, which is an analogue of \cite[Proposition 4.8]{Sha09Duke1}.
	
	\begin{proposition}\label{prop:algebraic_consequence}
		Let $l\in L$ such that $x_0=l\Lambda$. Suppose $\lambda_i\to\lambda$, then there exists $\gamma\in\Lambda$ such that
		\begin{equation}\label{eq:algebraic_consequence}
		\phi(s)l\gamma p_H \in V^{0-}(a),\quad \forall s\in I.
		\end{equation}
	\end{proposition}
	
	\begin{proof}
		By \eqref{eq:positive_limit_measure_on_singular_set} there exists a compact subset $C\subset N(H,W)\backslash S(H,W)$ and a constant $c_0 > 0$ such that $\lambda(\pi(C))>c_0$. We fix $0<\epsilon < c_0$, and apply \Cref{prop:nonfocusing_creterion} to obtain $\mathcal{D}$. We choose any relatively compact neighborhood $\Phi$ of $\mathcal{D}$, and obtain an $\mathcal{O}$ such that either (I) or (II) holds.\par 
		Since $\lambda_i\to\lambda$, there exists $i_0\in\N$ such that for all $i>i_0$, (I) does not hold. Therefore (II) holds for all $i>i_0$. In other words, there exists a sequence $\{\gamma_i \}$ in $\Lambda$ and a subinterval $J\subset I$ such that
		\begin{equation}
		\delta(s)a(t_i)\phi(s)l\gamma_ip_H\in \Phi,\quad \forall i>i_0,\,\forall s\in J.
		\end{equation}\par 
		By \Cref{thm:linear_stability}, we know that $\{\gamma_ip_H \}$ is bounded. Hence after passing to a subsequence, we may assume that there exists $\gamma\in\Lambda$ such that $\gamma_ip_H = \gamma p_H$ holds for all $i$. It follows that $a(t_i)\phi(s)l\gamma p_H$ remains bounded in $V$. This concludes the proof.
	\end{proof}
	
	Next we are able to obtain more algebraic information from \Cref{prop:algebraic_consequence}. First we show that the limiting process actually happens inside the $G$-orbit $G\cdot l\gamma p_H$.
	
	\begin{proposition}\label{prop:limit_on_G_orbit}
		Let the notations be as in \Cref{prop:algebraic_consequence}. Then for all but finitely many $s\in I=[a,b]$, there exists $\xi(s)\in P$ such that
		\begin{equation}\label{eq:limit_on_G_orbit}
		\lim_{t\to\infty}a(t)\phi(s)l\gamma p_H = \xi(s)\phi(s)l\gamma p_H.
		\end{equation}
	\end{proposition}
	
	\begin{proof}
		Denote $v=l\gamma p_H$. According to \Cref{prop:algebraic_consequence}, the limit on the left-hand side of \eqref{eq:limit_on_G_orbit} exists. We claim that the limit actually lies in the $G$-orbit $Gv$ for all but finitely many $s\in I$. \par 
		Consider the boundary $S=\partial(Gv)=\overline{Gv}\backslash Gv$. If $S$ is empty then the claim holds automatically. Now suppose that $S$ is non-empty, and that there exist infinitely many $s\in I$ such that $\lim_{t\to\infty}a(t)\phi(s)v$ is contained in $S$. Since $\phi$ is analytic, we have that for any $s\in I$, $\lim_{t\to\infty}a(t)\phi(s)v$ is contained in $S$. Hence in view of \eqref{eq:cor_basic_lemma1},
		\begin{equation}\label{eq:temp8}
		\phi(s)\in G(v,S,a),\quad\forall s\in J.
		\end{equation}
		Moreover, by \Cref{cor:basic_lemma1} there exists $\delta_0\in\Gamma^{+}(T)$ and $g_0\in G$ such that
		\begin{equation}\label{eq:temp9}
		G(v, S,a) \subset \bigsqcup_{w\in W^+(\delta_0, a)}Pw^{-1}Bg_0^{-1}.
		\end{equation}
		By \eqref{eq:temp8},\eqref{eq:temp9} and \Cref{lem:union_of_Schubert}(b), the image of $\widetilde{\phi}$ is contained in an unstable Schubert variety with respect to $a(t)$, which contradicts our assumption. \par 
		Hence for all but finitely many $s\in I$, there exists $\eta(s)\in G$ such that
		\begin{equation}\label{eq:temp15}
		\lim_{t\to\infty}a(t)\phi(s)v = \eta(s)\phi(s)v.
		\end{equation}
		Now fix any $s$ such that \eqref{eq:temp15} holds. Take $t_0>0$, and set $w=a(t_0)\phi(s)v$. Then
		\begin{equation}\label{eq:temp18}
		\lim_{t\to\infty}a(t)w = \eta(s)a(t_0)^{-1}w.
		\end{equation} 
		By taking $t_0$ large enough, we may assume that $\eta(s)a(t_0)^{-1}$ is contained in a small neighborhood of the neutral element in $G$. Let $F$ denote the stabilizer of $\eta(s)a(t_0)^{-1}w=\eta(s)\phi(s)v$ in $G$, and let $\mathfrak{f}$ be the Lie algebra of $F$. It is easy to see that $F$ contains $\{a(t)\}$.\par 
		Now the Lie algebra $\mathfrak{f}$ of $F$ is $\mathrm{Ad}\,a(t)$-invariant, and thus we have the following decomposition as a consequence of $a(t)$ being semisimple:
		\begin{equation}\label{eq:temp16}
		\mathfrak{g} = \mathfrak{f}^{\perp} \oplus \mathfrak{f},
		\end{equation}
		where $\mathfrak{f}^{\perp}$ is an $\mathrm{Ad}\,a(t)$-invariant subspace of $\mathfrak{g}$. \par 
		On the other hand, according to the eigenvalues of $\mathrm{Ad}\,a(t)$, we can decompose $\mathfrak{g}$ into
		\begin{equation}\label{eq:temp17}
		\mathfrak{g} = \mathfrak{g}^-\oplus \mathfrak{g}^0 \oplus \mathfrak{g}^+.
		\end{equation} \par 
		Combining the above two decompositions \eqref{eq:temp16}\eqref{eq:temp17}, we get
		\begin{equation}
		\mathfrak{g} = \mathfrak{g}^-\oplus \mathfrak{g}^0 \oplus (\mathfrak{g}^+\cap\mathfrak{f}^{\perp})\oplus (\mathfrak{g}^+\cap\mathfrak{f}).
		\end{equation}
		Hence there exist $X_s^{0-}\in \mathfrak{g}^0\oplus \mathfrak{g}^-$ and $X_s^+\in \mathfrak{g}^+\cap\mathfrak{f}^{\perp}$ such that
		\begin{equation}
		a(t_0)\eta(s)^{-1} \in \exp X_s^{0-}\exp X_s^{+}F.
		\end{equation}
		By \eqref{eq:temp18}, we have that $X_s^{+}=0$. Hence
		\begin{equation}
		a(t_0)\eta(s)^{-1}\in \exp X_s^{0-}F.
		\end{equation}
		Set $\xi(s)=\exp(-X_s^{0-})a(t_0)$, and one can verify that \eqref{eq:limit_on_G_orbit} holds.
	\end{proof}
	
	If we consider the slightly larger family of \emph{weakly} unstable Schubert varieties, and further assume that the image of $\widetilde{\phi}$ is not contained in any weakly unstable Schubert variety, then we could obtain the following refinement of \Cref{prop:limit_on_G_orbit}.
	
	\begin{proposition}\label{prop:closed_orbit_reductive_stablizer}
		In the situation of \Cref{prop:algebraic_consequence}, further assume that the image of $\widetilde{\phi}$ is not contained in any weakly unstable Schubert variety of $G/P$ with respect to $a(t)$. Then the orbit $G\cdot l\gamma p_H$ is closed, and the stabilizer of $l\gamma p_H$ in $G$ is reductive.
	\end{proposition}
	
	\begin{proof}
		Write $v=l\gamma p_H$. Suppose that $Gv$ is not closed, then the boundary $S=\partial(Gv)$ is non-empty. By \Cref{prop:basic_lemma2} there exists $\delta_0\in\Gamma^{+}(T)$ and $g_0\in G$ such that
		\begin{equation}\label{eq:temp1}
		G(v, V^{0-}(a)) \subset \bigsqcup_{w\in W^{0+}(\delta_0, a)}Pw^{-1}Bg_0^{-1}.
		\end{equation}
		Also by \eqref{eq:algebraic_consequence} we know
		\begin{equation}\label{eq:temp2}
		\phi(s) \in G(v, V^{0-}(a)),\quad \forall s\in I.
		\end{equation}
		By \eqref{eq:temp1}, \eqref{eq:temp2} and \Cref{lem:union_of_Schubert}(c), the image of $\widetilde{\phi}$ is contained in a weakly unstable Schubert variety, which contradicts our assumption on $\phi$. \par 
		Therefore $Gv$ is closed, i.e. $G\cdot l\gamma p_H$ is closed. By Matsushima's criterion, the stabilizer of $l\gamma p_H$ in $G$ is reductive.
	\end{proof}
	
	The following proposition describes the obstructions to equidistribution. (C.f. \cite[Theorem 6.1]{SY16}.)
	
	\begin{proposition}\label{prop:algebraic_obstruction}
		Suppose that the image of $\widetilde{\phi}$ is not contained in any unstable Schubert variety of $G/P$ with respect of $a(t)$, and that $\lambda_i\to\lambda$. Then there exists $g\in G$ and an algebraic subgroup $F$ of $L$ containing $\{a(t)\}$ such that $Fgl\Lambda$ is closed and admits a finite $F$-invariant measure, and that
		\begin{equation}\label{eq:stuck_in_F}
		\phi(s) \in P(F\cap G)g,\quad\forall s\in I.
		\end{equation} \par 
		Furthermore, if the image of $\widetilde{\phi}$ is not contained in any weakly unstable Schubert variety, then we can choose $F$ such that $F\cap G$ is reductive.
	\end{proposition}
	
	\begin{proof}
		Let $\xi(s)$ be defined as in \Cref{prop:limit_on_G_orbit}. Since the right hand side of \eqref{eq:stuck_in_F} is left $a(t)$-invariant, without loss of generality we may replace $\phi(s)$ with $a(t_0)\phi(s)$ for some large $t_0>0$, and assume that $\xi(s)$ lies in a small neighborhood of $e$ in $G$, for all $s\in I$. Hence we may take $\xi(s)\in P$. \par 
		Fix any $s_0\in I$. Let $g = \xi(s_0)\phi(s_0)$ and $v=l\gamma p_H$. We set $F= \mathrm{Stab}_L(gv)=gl\gamma N_L^1(H)\gamma^{-1}l^{-1}g^{-1}$. By \Cref{prop:limit_on_G_orbit} we have $\{a(t)\}\subset F$. Since $\Lambda\cdot p_H$ discrete, $N_L^1(H)\cdot\Lambda$ is closed. Hence $Fgl\Lambda$ is also closed. \par
		Now the Lie algebra $\mathfrak{f}$ of $F$ is $\mathrm{Ad}\,a(t)$-invariant, and thus we have the following decomposition as a consequence of $a(t)$ being semisimple:
		\begin{equation}\label{eq:temp6}
		\mathfrak{g} = \mathfrak{f}^{\perp} \oplus \mathfrak{f},
		\end{equation}
		where $\mathfrak{f}^{\perp}$ is an $\mathrm{Ad}\,a(t)$-invariant subspace of $\mathfrak{g}$. \par 
		On the other hand, according to the eigenvalues of $\mathrm{Ad}\,a(t)$, we can decompose $\mathfrak{g}$ into
		\begin{equation}\label{eq:temp7}
		\mathfrak{g} = \mathfrak{g}^-\oplus \mathfrak{g}^0 \oplus \mathfrak{g}^+.
		\end{equation} \par 
		Combining the above two decompositions \eqref{eq:temp6}\eqref{eq:temp7}, we get
		\begin{equation}
		\mathfrak{g} = \mathfrak{g}^-\oplus \mathfrak{g}^0 \oplus (\mathfrak{g}^+\cap\mathfrak{f}^{\perp})\oplus (\mathfrak{g}^+\cap\mathfrak{f}).
		\end{equation}
		Hence for all $s$ near $s_0$, there exist $X_s^{0-}\in \mathfrak{g}^0\oplus \mathfrak{g}^-$ and $X_s^+\in \mathfrak{g}^+\cap\mathfrak{f}^{\perp}$ such that
		\begin{equation}\label{eq:temp3}
		\xi(s_0)\phi(s)g^{-1} \in \exp X_s^{0-}\exp X_s^+F.
		\end{equation} \par 
		Since $a(t_i)\phi(s)v$ converge in $V$ as $i\to\infty$, by \Cref{prop:limit_on_G_orbit} we know that $a(t_i)\phi(s)g^{-1}F$ converge in $G/F$ as $i\to\infty$. It follows that
		\begin{equation}\label{eq:temp4}
		X_s^+ = 0,\quad \forall s\in I.
		\end{equation} \par 
		Since $X_s^{0-}\in \mathfrak{g}^0\oplus \mathfrak{g}^-$, we have
		\begin{equation}\label{eq:temp5}
		\exp X_s^{0-}\in P.
		\end{equation} \par 
		Combining \eqref{eq:temp3}\eqref{eq:temp4}\eqref{eq:temp5} we get
		\begin{equation}
		\phi(s) \in PFg,
		\end{equation}
		for all $s\in I$. This implies \eqref{eq:stuck_in_F}. Moreover, by \cite[Theorem 2.3]{Sha91}, there exists a subgroup $F_1$ of $F$ containing all $\mathrm{Ad}$-unipotent one-parameter subgroups of $L$ contained in $F$ such that $F_1gl\Lambda$ admits a finite $F_1$-invariant measure. Since $F$ contains $\{a(t)\}$, $P$ contains the central torus of $F$. Hence $PFg=PF_1g$, and we may replace $F$ by $F_1$. \par 
		If we further assume that the image of $\widetilde{\phi}$ is not contained in any weakly unstable Schubert variety, then by \Cref{prop:closed_orbit_reductive_stablizer} we know that the stabilizer of $l\gamma p_H$ in $G$ is reductive, i.e. $g^{-1}Fg\cap G$ is reductive. Hence $F\cap G$ is also reductive. 
	\end{proof}
	
	\subsection{Lifting of obstructions and proof of equidistribution results}
	
	In this section, we show that the conditions in \Cref{thm:equidistribution1} are preserved under projections. This enables us to use induction to prove the equidistribution results.
	
	\begin{lemma}\label{lem:functoriality}
		Let $G$ be a connected semisimple real algebraic group, and $p\colon G\rightarrow\overline{G}$ be a surjective homomorphism. Let $a(t)$ be a multiplicative one-parameter subgroup of $G$, and $\overline{a(t)}$ be its image in $\overline{G}$. Suppose that $\overline{a(t)}$ is non-trivial. Define (weakly) unstable Schubert varieties and partial flag subvarieties of $\overline{G}/\overline{P}$ with respect to $\overline{a(t)}$, $\overline{T}$ and $\overline{B}$. Then the preimage of any unstable (resp. weakly unstable) Schubert subvariety of $\overline{G}/\overline{P}$ with respect to $\overline{a(t)}$ is an unstable (resp. weakly unstable) Schubert subvariety of $G/P$ with respect to $a(t)$.
	\end{lemma}
	
	\begin{proof}
		Let $X_{\overline{w}}$ be an unstable Schubert subvariety of $\overline{G}/\overline{P}$, where $\overline{w}\in W^{\overline{P}}$ such that $(\overline{\delta},\overline{a}^{\overline{w}})\geq 0$ for some $\overline{\delta}\in\Gamma^{+}(\overline{T})$. Let $G_1$ denote the kernel of $p$, and we have $W_G=W_{G_1}\times W_{\overline{G}}$. Let $w_0$ denote the unique maximal element in $W^{P_1}$. Then the preimage of $X_{\overline{w}}$ is $X_{(w_0,\overline{w})}$. Now it remains to check instability. We note that the Killing form on $\mathfrak{g}$ is the sum of the Killing forms on $\mathfrak{g}_1$ and $\overline{\mathfrak{g}}$. Hence we consider the lifted multiplicative one-parameter subgroup $(e,\overline{\delta})\in\Gamma^{+}(T)$, and use it to check that $X_{(w_0,\overline{w})}$ is unstable. \par 
		The same proof also works for weakly unstable Schubert varieties.
	\end{proof}
	
	We now proceed to the equidistribution results. Recall that $l\in L$ such that $x_0=l\Lambda$, and $\lambda_i$ are probability measures on $L/\Lambda$ as defined in \eqref{eq:def_lambda_i}. 
	
	\begin{proposition}\label{prop:twisted_limit_measure_Haar}
		Let $\phi$ be an analytic curve on $G$ such that the following two conditions hold:
		\begin{enumerate}[(a)]
			\item the image of $\widetilde{\phi}$ is not contained in any unstable Schubert variety of $G/P$ with respect to $a(t)$;
			\item For any $g\in G$ and any proper algebraic subgroup $F$ of $L$ containing $\{a(t)\}$ such that $Fgx_0$ is closed and admits a finite $F$-invariant measure, the image of $\phi$ is not contained in $P(F\cap G)g$.
		\end{enumerate}
		Suppose that $\lambda_i\to\lambda$ in the weak-* topology, then $\lambda$ is the unique $L$-invariant probability measure on $L/\Lambda$.
	\end{proposition}
	
	\begin{proof}
		By \Cref{prop:algebraic_obstruction}, there exists an algebraic subgroup $F$ of $L$ such that \eqref{eq:stuck_in_F} holds. Then condition (b) implies that $F=G$, and thus $G$ fixes $l\gamma p_H$. Arguing as in the proof of \cite[Theorem 5.6]{Sha09Invention}, we know that $L=N_L^1(H)$, i.e. $H$ is normal in $L$. \par 
		Now we can prove the theorem by induction on the number of simple factors in $L$. If $L$ is simple, then we have $H=L$, and $\lambda$ is $H=L$-invariant. For the inductive step, we consider the natural quotient map $p\colon L\rightarrow L/H$. For any subset $E\subset L$, let $\overline{E}$ denote its image under the quotient map. By \Cref{lem:functoriality}, $\overline{\phi(I)}$ is not contained in any unstable Schubert variety with respect to $\overline{a(t)}$. Hence $\overline{\phi}$ still satisfies condition (a). One can also verify that $\overline{\phi}$ still satisfies condition (b). Indeed, if the image of $\overline{\phi}$ is contained in $\overline{P}(F_0\cap\overline{G})\overline{g}$ for some $F_0\subsetneq\overline{L}$ such that $F_0\overline{g}\overline{x_0}$ is closed, then the image of $\phi$ is contained in $P(p^{-1}(F_0)\cap G)g$ and $p^{-1}(F_0)gx_0$ is also closed. \par 
		Now both conditions still hold for the projected curve $\overline{\phi}$. By inductive hypothesis we know that the projected measure $\overline{\lambda}$ is the $L/H$-invariant measure on $L/H\Lambda$. In addition, we already know that $\lambda$ is $H$-invariant. Therefore $\lambda$ is $L$-invariant.
	\end{proof}
	
	\begin{corollary}\label{cor:limit_measure_Haar}
		Let $\phi$ be an analytic curve satisfying (a) and (b) in \Cref{prop:twisted_limit_measure_Haar}. Let $\mu_i$ be the probability measure on $L/\Lambda$ as defined in \eqref{eq:mu_i}. Suppose that $\mu_i\to\mu$ with respect to the weak-* topology, then $\mu$ is the unique $L$-invariant probability measure on $L/\Lambda$.
	\end{corollary}
	
	\begin{proof}
		The deduction of \Cref{cor:limit_measure_Haar} from \Cref{prop:twisted_limit_measure_Haar} is analogous to the proof of \cite[Corollary 5.7]{Sha09Invention}.
	\end{proof}
	
	Parallel to \Cref{prop:twisted_limit_measure_Haar} and \Cref{cor:limit_measure_Haar}, the following results could be proved with the same arguments.
	
	\begin{proposition}\label{prop:twisted_limit_measure_Haar1}
		Let $\phi$ be an analytic curve on $G$ such that the following two conditions hold:
		\begin{enumerate}[(A)]
			\item the image of $\widetilde{\phi}$ is not contained in any weakly unstable Schubert variety of $G/P$ with respect to $a(t)$;
			\item For any $g\in G$ and any proper algebraic subgroup $F$ of $L$ containing $\{a(t)\}$ such that $Fgx_0$ is closed and admits a finite $F$-invariant measure and that $F\cap G$ is reductive, the image of $\phi$ is not contained in $P(F\cap G)g$.
		\end{enumerate}
		Suppose that $\lambda_i\to\lambda$ in the weak-* topology, then $\lambda$ is the unique $L$-invariant probability measure on $L/\Lambda$.
	\end{proposition}
	
	\begin{corollary}\label{cor:limit_measure_Haar1}
		Let $\phi$ be an analytic curve satisfying (A) and (B) in \Cref{prop:twisted_limit_measure_Haar1}. Let $\mu_i$ be the probability measure on $L/\Lambda$ as defined in \eqref{eq:mu_i}. Suppose that $\mu_i\to\mu$ with respect to the weak-* topology, then $\mu$ is the unique $L$-invariant probability measure on $L/\Lambda$.
	\end{corollary}
	
	Now we are ready to prove the main theorems in \Cref{subsect:equidistribution}.
	
	\begin{proof}[Proof of \Cref{thm:equidistribution0}]
		If \eqref{eq:equidistribution0} fails to hold, then there exist $\epsilon>0$ and a sequence $t_i\to\infty$ such that for each $i$,
		\begin{equation}
		\left\vert \frac{1}{b-a}\int_{a}^{b}f(a(t_i)\phi(s)x_0)\,\mathrm{d}s - \int_{L/\Lambda}f\,\mathrm{d}\mu_{L/\Lambda}
		\right\vert \geq \epsilon.
		\end{equation}
		In view of \eqref{eq:mu_i} and \Cref{cor:weak_limit_exists}, this statement contradicts \Cref{cor:limit_measure_Haar}.
	\end{proof}
	
	\begin{proof}[Proof of \Cref{thm:equidistribution1}]
		If \eqref{eq:equidistribution1} fails to hold, then there exist $\epsilon>0$ and a sequence $t_i\to\infty$ such that for each $i$,
		\begin{equation}
		\left\vert \frac{1}{b-a}\int_{a}^{b}f(a(t_i)\phi(s)x_0)\,\mathrm{d}s - \int_{L/\Lambda}f\,\mathrm{d}\mu_{L/\Lambda}
		\right\vert \geq \epsilon.
		\end{equation}
		In view of \eqref{eq:mu_i} and \Cref{cor:weak_limit_exists}, this statement contradicts \Cref{cor:limit_measure_Haar1}.
	\end{proof}

	\section{Grassmannians and Schubert varieties}\label{sect:Grassmannian_Schubert}
	In this section we consider the special case where $G = L = \SL_{m+n}(\R)$, and $\Lambda = \SL_{m+n}(\Z)$. Define
	\[
	a(t) = 
	\begin{bmatrix}
	t^nI_m & \\
	& t^{-m}I_n
	\end{bmatrix}.
	\]
	Then $\{a(t)\}$ is a multiplicative one-parameter subgroup of $G$. In this section, all the unstable and weakly unstable Schubert varieties are with respect to this $a(t)$. Let $P$ be the parabolic subgroup associated with $\{a(t)\}$. We have
	\begin{equation}
	P = \left\{ 
	\begin{bmatrix}
	A & \mathbf{0}\\
	C & D
	\end{bmatrix}\in\SL_{m+n}(\R) \colon
	A\in M_{m\times m}(\R),\,C\in M_{n\times m}(\R),\,D\in M_{n\times n}(\R)
	\right\}.
	\end{equation}
	Hence the partial flag variety $G/P$ coincide with $\mathrm{Gr}(m,m+n)$, the Grassmannian of $m$-dimensional subspaces of $\R^{m+n}$. It is an irreducible projective variety of dimension $mn$.\par
	
	\subsection{Schubert cells and Schubert varieties}\label{subsect:Schubert}
	Let $B$ be the Borel subgroup of lower triangular matrices in $G$, and $T$ the group of diagonal matrices in $G$. The Weyl group $W=N_G(T)/Z_G(T)$ is isomorphic to $S_{m+n}$, the permutation group on $m+n$ elements. The Weyl group $W_P$ of $P$ is isomorphic to $S_{m}\times S_{n}$, and the set $W^P$ of minimal length coset representatives of $W/W_P$ consists of the permutations $w=(w_1,\cdots,w_{m+n})$ such that $w_1<\cdots<w_m$ and $w_{m+1}<\cdots<w_{m+n}$. We identify $w$ in $W^P$ with the subset $I_w=\{w_1, \cdots, w_m \}$ of $\{1,2,\cdots, m+n \}$. The cosets $wP$ are exactly the $T$-fixed points of $G/P$. The Schubert cell $C_w$ is by definition $BwP$, and the Schubert variety $X_w$ is defined to be $\overline{BwP}$, the closure of $C_w$ in $G/P$. For $w,w'\in W^P$, $w'\in X_w$ if and only if $w'\leq w$ in the Bruhat order. We note that the Bruhat order here is the order on the tuples $(w_1,\cdots ,w_m)$ given by
	\[
	(w_i)\leq (v_i) \iff w_i\leq v_i,\forall 1\leq i\leq m.
	\]
	The dimension of $X_w$ is given by $l(w)$, which equals $\sum_{k=1}^{m} (w_k-k)$.\par
	
	The definitions above coincide with the classical definitions. For $1\leq k \leq m+n$, let $F_k$ be the standard $k$-dimensional subspace of $\R^{m+n}$ spanned by $\{e_1,\cdots, e_k \}$. We have the complete flag of subspaces
	\begin{equation}
	0=F_0\subset F_1\subset F_2\cdots\subset F_{m+n-1}\subset F_{m+n}=\R^{m+n}.
	\end{equation}
	For an $m$-dimensional subspace $V\in\mathrm{Gr}(m,m+n)$ of $\R^{m+n}$, consider the intersections of the subspace with the flag:
	\begin{equation}
	0\subset (F_1\cap V)\subset (F_2\cap V)\cdots\subset (F_{m+n-1}\cap V)\subset W.
	\end{equation}
	For $w\in W^P$, we have a tuple $(w_1,\cdots, w_m)$, and the Schubert cell $C_w$ has the following description:
	\begin{equation}
	C_w = \left\{V\in\mathrm{Gr}(m,m+n)\colon\dim(V\cap F_{w_{k}})=k;\;\dim(V\cap F_l)<k,\,\forall l<w_k \right\}.
	\end{equation}
	In other words, the tuple $(w_1,\cdots, w_m)$ gives the indices where the dimension jumps.\par
	Similarly, the Schubert variety $X_w$ has the following description:
	\begin{equation}
	X_w = \left\{V\in\mathrm{Gr}(m,m+n)\colon\dim(V\cap F_{w_{k}})\geq k,\; 1\leq k\leq m \right\}.
	\end{equation}\par
	Now it is easy to see that
	\begin{equation}
	X_w = \bigsqcup_{w'\leq w}C_{w'}.
	\end{equation}
	Hence the Schubert cells give a stratification of the Grassmannian variety.\par
	
	\begin{example}\label{ex:Schubert}
		\begin{enumerate}[(1)]
			\item For $m = 1$, the Grassmannian $\mathrm{Gr}(1,n)$ is just the projective space $\R\mathbb{P}^n$, and the Schubert varieties form a flag of linear subspaces $X_0\subset X_1\subset\cdots\subset X_n$, where $X_j\cong \R\mathbb{P}^{j}$.
			\item For $m=n=2$ one gets the following poset of Schubert varieties in $\mathrm{Gr}(2,4)$:
			
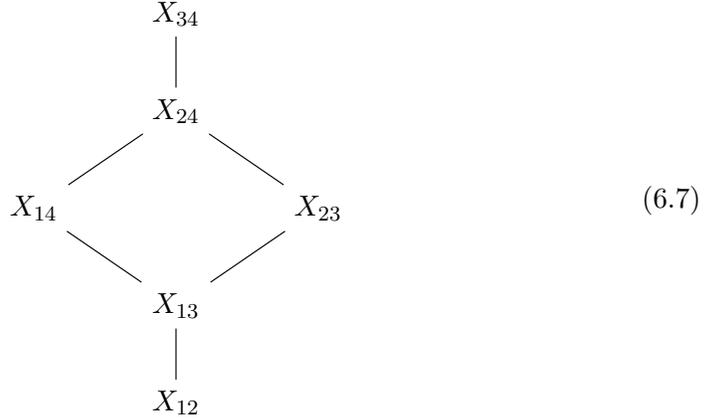
\begin{equation}
			\begin{tikzcd}
			&X_{34} \arrow[d, dash] \\
			&X_{24} \arrow[dr, dash] \\
			X_{14} \arrow[ur, dash] &  &X_{23} \arrow[dl, dash] \\
			&X_{13} \arrow[ul, dash] \\
			&X_{12} \arrow[u, dash]
			\end{tikzcd}
			\end{equation}
			where $X_{12}$ is one single point, and $X_{34}$ is $\mathrm{Gr}(2,4)$.
		\end{enumerate}
	\end{example}
	
	\subsection{Pencils}
	The main goal of this section is to show that maximal (weakly) constraining pencils coincide with maximal (weakly) unstable Schubert varieties in the Grassmannian case, and hence the latter is a natural generalization to all partial flag varieties.\par 
	
	Given a real vector space $W \subsetneq \R^{m+n}$, and an integer $r\leq m$, we recall from \Cref{def:pencil} that the pencil $\mathfrak{P}_{W,r}$ is the set 
	\[
	\{V\in\mathrm{Gr}(m,m+n)\colon \mathrm{dim}(V\cap W)\geq r\}.
	\]
	Denote $d=\dim W$. Let $w\in W^P$ be the element such that $(w_1, \cdots, w_m)$ is the tuple
	\[
	(d-r+1, \cdots, d,r+1,\cdots, m).
	\]
	One can verify that the pencil $\mathfrak{P}_{W,r}$ is the Schubert variety $gX_w$, where $g$ is an element in $\SL_{m+n}(\R)$ such that $W = g\cdot F_d$. The pencil is called constraining (resp. weakly constraining) if the inequality \eqref{eq:constraining_pencil} (resp. \eqref{eq:weakly_constraining_pencil}) holds. \par 
	
	On the other hand, we recall that the Schubert variety $X_w$ is unstable (resp. weakly unstable) if there exists a non-trivial multiplicative one-parameter subgroup $\delta$ in $\Gamma^+(T)$ such that $(\delta, a^w)>0$ (resp. $\geq 0$). Let $\Delta$ be the element in the Lie algebra $\mathfrak{t}$ of $T$ such that $\delta(t) = \exp(\log t\cdot\Delta)$. Then $\Delta$ could be written as $\mathrm{diag}(t_1,t_2,\cdots, t_{m+n})$, where $t_1\geq t_2 \geq\cdots\geq t_{m+n}$ and $\sum t_i=0$. Hence in the case of Grassmannian we have the following criterion of stability.
	
	\begin{lemma}\label{lem:Schubert_stability_criterion}
		Let $w$ be an element in $W^P$, then the corresponding Schubert variety $X_w$ is unstable (resp. weakly unstable) if and only if the following system is soluble:
		\begin{gather}
		t_1\geq\cdots\geq t_k>0\geq t_{k+1}\geq\cdots\geq t_{m+n}   \label{eq:order_t}  \\  
		\sum_{i=1}^{m+n} t_i=0  \label{eq:sum_equals_zero} \\
		\sum_{j=1}^{m} t_{w_j} > 0 \;(\text{resp. }\sum_{j=1}^{m} t_{w_j} \geq 0) \label{eq:combinatorial_inequality}
		\end{gather}
	\end{lemma}
	
	\begin{example}[$m=n=2$]
		We continue with \Cref{ex:Schubert}(2). If $w=(14)$, then we can take $t_1 = 3, t_2=t_3=t_4=-1$, which gives $t_1+t_4>0$. Hence by \Cref{lem:Schubert_stability_criterion} we have $X_{14}$ is unstable. Similarly we can show that $X_{23}$ is unstable by taking $t_1=t_2=t_3=1, t_4=-3$.\par 
		When $w=(24)$, $t_2+t_4 \geq 0$ is soluble as we can take $t_1=t_2=1,t_3=t_4=-1$. However, $t_2+t_4>0$ is insoluble. Indeed, suppose $t_2+t_4>0$, then $t_1+t_3\geq t_2+t_4>0$, and it follows that $t_1+t_2+t_3+t_4>0$, which contradicts \eqref{eq:sum_equals_zero}. Therefore we conclude that $X_{24}$ is weakly unstable but not unstable.
	\end{example}
	
	Now we are ready for the main results of this section.
	
	\begin{proposition}\label{prop:pencil_in_Schubert}
		Every constraining (resp. weakly constraining) pencil is an unstable (resp. weakly unstable) Schubert variety of $\mathrm{Gr}(m,m+n)$.
	\end{proposition}
	
	\begin{proof}
		Let $\mathfrak{P}_{W,r}$ be a constraining pencil, and thus by definition we have
		\begin{equation}\label{eq:constraining_pencil_recall}
		\frac{d}{r} < \frac{m+n}{m},
		\end{equation}
		where $d = \dim W$. Then $\mathfrak{P}_{W,r}=gX_w$, where $g\in G$ and $w\in W^P$ such that
		\begin{equation}
		(w_1,\cdots, w_m) = (d-r+1,\cdots,d,n+r+1,\cdots,m+n).
		\end{equation}
		Now set $t_1 = \cdots = t_d = m + n - d$ and $t_{d+1} = \cdots = t_{m+n} = -d$. It is clear that \eqref{eq:order_t} and \eqref{eq:sum_equals_zero} are satisfied. Moreover,
		\begin{equation}
		\begin{split}
		\sum_{j=1}^{m} t_{w_j} &= r(m+n-d) - (m-r)d\\
		&= r(m+n) - md\\
		&= mr\left(\frac{m+n}{m}-\frac{d}{r}\right)\\
		&> 0.
		\end{split}
		\end{equation}
		Hence \eqref{eq:combinatorial_inequality} also holds. Therefore, by \Cref{lem:Schubert_stability_criterion} we conclude that $\mathfrak{P}_{W,r}$ is an unstable Schubert variety. The same proof also works for weakly constraining pencils.
	\end{proof}
	
	\begin{proposition}\label{prop:Schubert_in_pencil}
		Every unstable (resp. weakly unstable) Schubert variety of $\mathrm{Gr}(m,m+n)$ is contained in a constraining (resp. weakly constraining) pencil.
	\end{proposition}
	
	\begin{proof}
		Let $X_w$ be an unstable Schubert variety and consider the set $I_w=\{w_1,\cdots, w_m\}$. Let $J_w$ be the subset of $I_w$ consisting of the elements with jump, that is, $w_k$ is contained in $J_w$ if and only if $w_{k+1}-w_{k}>1$. Here we set $w_{m+1}=0$. Notice that for any $w_k\in J_w$, if we set $W=F_{w_k}$ and $r=k$, then $X_w$ is contained in the pencil $\mathfrak{P}_{W,r}$. Now it suffices to show that there exists $w_k\in J_w$ such that
		\begin{equation}\label{eq:slope_inequality}
		\frac{w_k}{k} < \frac{m+n}{m}.
		\end{equation}
		Actually, the function $k\mapsto{w_k}/{k}$ achieves its minimum at some $k$ such that $w_k\in J_w$. Hence it suffices to prove the following claim.
		\begin{claim}
			There exists $1\leq k \leq m$ such that \eqref{eq:slope_inequality} holds.
		\end{claim}
		\par 
		We prove the claim by contradiction. Suppose that for any $1\leq k \leq m$ we have 
		\begin{equation}\label{eq:contradiction_slope_inequality}
		\frac{w_k}{k} \geq \frac{m+n}{m}.
		\end{equation}\par 
		For $1\leq i\leq m+n$, consider the auxiliary function
		\begin{equation}
		g(i) =
		\begin{cases}
		-m & i\notin I_w;\\
		n & i\in I_w.
		\end{cases}
		\end{equation}
		For any $1\leq i< m+n$, let $w_k$ be the largest element in $I_w$ such that $w_k\leq i$ (and set $w_k=0$ if $i<w_1$). As a consequence of \eqref{eq:contradiction_slope_inequality}, we have
		\begin{equation}
		\begin{split}
		\sum_{j=1}^{i}g(i) &\leq \sum_{j=1}^{w_k}g(i)\\
		&= -m(w_k-k) + nk\\
		&=(m+n)k-mw_k\\
		&\leq 0.\quad\text{By }\eqref{eq:contradiction_slope_inequality}
		\end{split}
		\end{equation}
		It is also clear that
		\begin{equation}
		\sum_{j=1}^{m+n}g(i) = 0.
		\end{equation}\par 
		Since $X_w$ is unstable, we may find $t_1,\cdots,t_{m+n}$ satisfying \eqref{eq:order_t}\eqref{eq:sum_equals_zero}\eqref{eq:combinatorial_inequality}. Denote
		\begin{gather}
		A = \sum_{i\in I_w}t_i;\\
		B = \sum_{i\notin I_w}t_i.
		\end{gather}
		Then $A>0$ and $A+B=0$ by \eqref{eq:sum_equals_zero}\eqref{eq:combinatorial_inequality}. Hence $B<0$, and $nA-mB>0$.\par 
		On the other hand, summation by parts leads to
		\begin{equation}
		\begin{split}
		nA-mB &= n\sum_{i\in I_w}t_i - m\sum_{i\notin I_w}t_i\\
		&= \sum_{i=1}^{m+n}g(i)t_i\\
		&= \sum_{i=1}^{m+n-1}\left[(t_i-t_{i+1})\sum_{j=1}^{i}g(j)\right]+t_{m+n}\sum_{j=1}^{m+n}g(j)\\
		&= \sum_{i=1}^{m+n-1}\left[(t_i-t_{i+1})\sum_{j=1}^{i}g(j)\right]\\
		&\leq 0.
		\end{split}
		\end{equation}
		This is a contradiction.\par 
		Therefore we have proved the claim, and thus $\mathfrak{P}_{W,r}$ is a constraining pencil containing the Schubert variety $X_w$. The same proof works for weakly unstable Schubert varieties.
	\end{proof}
	
	Combining \Cref{prop:pencil_in_Schubert} and \Cref{prop:Schubert_in_pencil}, we conclude the following.
	
	\begin{theorem}\label{thm:pencil_equal_Schubert}
		Let $E$ be any subset of $\mathrm{Gr}(m,m+n)\cong G/P$. Then $E$ is contained in an unstable (resp. weakly unstable) Schubert variety with respect to $a(t)$ if and only if $E$ is contained in a constraining (resp. weakly constraining) pencil.
	\end{theorem}
	
	\subsection{Young diagrams}
	In this section, we will give a combinatorial description of pencils and (weakly) constraining pencils, using Young diagrams. This will enable us to quickly see whether a Schubert variety is a pencil, and whether a pencil is (weakly) constraining. The readers are referred to Fulton's book \cite{Ful97} for more details.\par 
	
	A \emph{partition} is a sequence of integers $\lambda=(\lambda_1,\cdots,\lambda_m)$ such that $\lambda_1\geq\cdots\geq\lambda_m\geq 0$. Let $\Pi_{m,n}$ denote the set of partitions such that $\lambda_1 \leq n$. A \emph{Young diagram} is a set of boxes arranged in a left justified array, such that the row lengths weakly decrease from top to bottom. To any partition $\lambda$ we associate the Young diagram $D_{\lambda}$ whose $i$-th row contains $\lambda_i$ boxes. An \emph{outside corner} of the Young diagram $D_{\lambda}$ is a box in $D_{\lambda}$ such that removing the box we still get a Young diagram.
	
	\begin{example}\label{ex:Young_diagram}
		Let $m = 3$, $n = 5$, and $\lambda=(4,3,1)\in\Pi_{m,n}$. The Young diagram $D_{\lambda}$ fits inside an $m\times n$ rectangle.
		\[ 
		\ytableausetup{nosmalltableaux}
		\ytableausetup{centertableaux}
		\ytableaushort
		{\none\none\none\bullet, \none\none\bullet,\bullet}
		*{5,5,5}
		*[*(yellow)]{4,3,1}
		\]
		There are three outside corners, which are marked with a dot in the diagram.
	\end{example}\par
	
	Given $\lambda\in\Pi_{m,n}$, the associated Schubert variety $X_{\lambda}\subset\mathrm{Gr}(m,m+n)$ is defined by the conditions
	\begin{equation}
	\dim(V\cap F_{n+i-\lambda_i}) \geq i,\quad 1\leq i\leq m.
	\end{equation}
	Actually we only need outside corners to define $X_{\lambda}$; the pairs $(i,\lambda_i)$ which are not outside corners are redundant. (See \cite[Exercise 9.4.18]{Ful97}.) Therefore, we have the following lemma.
	
	\begin{lemma}
		Given $\lambda\in\Pi_{m,n}$, the Schubert variety $X_\lambda$ is a pencil if and only if the Young diagram $D_{\lambda}$ has only one outside corner.
	\end{lemma}
	
	The Schubert variety given by \Cref{ex:Young_diagram} is not a pencil, as the Young diagram has three outside corners. However, every Schubert variety can be written as an intersection of pencils.\\
	
	One can also recognize constraining and weakly constraining pencils with the help of Young diagrams.\par 
	
	For an $m\times n$ rectangle, we draw the diagonal connecting the northeast and the southwest of the rectangle. A \emph{node} is a vertex of a box. We call a node \emph{unstable} if it is lying below the diagonal, and \emph{weakly unstable} if it is lying on or below the diagonal. See \Cref{fig:nodes} for an example.\par 
	
	Now we can reformulate the definition of constraining and weakly constraining pencils.
	
	\begin{lemma}\label{lem:pencil_criterion}
		A pencil $X_{\lambda}$ is constraining (resp. weakly constraining) if and only if the bottom-right vertex of the outside corner of $D_{\lambda}$ is an unstable (resp. weakly unstable) node.
	\end{lemma}
	
	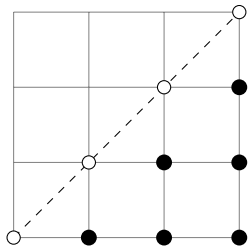
\begin{figure}
		\centering
		\[
		\begin{tikzpicture}
		\draw [help lines] (0,0) grid (3,3);
		\draw [dashed] (0,0) -- (3,3);
		\node[circle,draw=black, fill=white, inner sep=0pt,minimum size=5pt] (b) at (0,0) {};
		\node[circle,draw=black, fill=white, inner sep=0pt,minimum size=5pt] (b) at (1,1) {};
		\node[circle,draw=black, fill=white, inner sep=0pt,minimum size=5pt] (b) at (2,2) {};
		\node[circle,draw=black, fill=white, inner sep=0pt,minimum size=5pt] (b) at (3,3) {};
		\fill [black] (1,0) circle (3pt);
		\fill [black] (2,0) circle (3pt);
		\fill [black] (3,0) circle (3pt);
		\fill [black] (2,1) circle (3pt);
		\fill [black] (3,1) circle (3pt);
		\fill [black] (3,2) circle (3pt);
		\end{tikzpicture}
		\]
		\caption{Unstable and weakly unstable nodes in a $3\times3$ rectangle. The black nodes are unstable, while the white nodes are weakly unstable but not unstable.} \label{fig:nodes}
	\end{figure}
	
	\begin{example}
		Let $m=2$ and $n=3$. By \Cref{lem:pencil_criterion} there are $5$ constraining pencils: $X_{12},X_{15},X_{23},X_{25}$ and $X_{34}$. Among those $X_{25}$ and $X_{34}$ are the maximal ones, and they give the obstruction to non-divergence.
		\[
		\ytableausetup{nosmalltableaux}
		\ytableausetup{centertableaux}
		\ytableaushort
		{\none\bullet, \none}
		*{3,3}
		*[*(yellow)]{2,0}
		\quad\quad
		\ytableausetup{nosmalltableaux}
		\ytableausetup{centertableaux}
		\ytableaushort
		{\none, \bullet}
		*{3,3}
		*[*(yellow)]{1,1}
		\]
	\end{example}
	As noted in \Cref{rem:pencils}, the weakly constraining pencils coincide with the constraining pencils in the case that $m$ and $n$ are coprime. This also follows from the simple observation that there are no nodes lying on the diagonal of $D_{\lambda}$.

	\bibliographystyle{alpha}
	\bibliography{references}
	
\end{document}